\documentclass[12pt]{amsart}

\usepackage[all,arc]{xy}
\usepackage{enumerate, amsmath, amsthm, amsfonts, amssymb, xy,  mathrsfs, graphicx, paralist, fancyvrb,ytableau,tikz,caption,verbatim}
\usepackage{xcolor}
\usepackage[margin=1in]{geometry} 
\usepackage[bookmarks, colorlinks=true, linkcolor=blue, citecolor=blue, urlcolor=blue]{hyperref}

\newtheorem{thm}{Theorem}[section]
\newtheorem{cor}[thm]{Corollary}
\newtheorem{prop}[thm]{Proposition}
\newtheorem{lem}[thm]{Lemma}

\theoremstyle{definition}
\newtheorem{defn}[thm]{Definition}

\newtheorem{exmp}[thm]{Example}

\newtheorem{notn}[thm]{Notation}

\theoremstyle{remark}
\newtheorem{rmk}[thm]{Remark}

\makeatletter
\let\c@equation\c@thm
\makeatother

\newcommand{\cA}{\mathcal{A}}

\newcommand{\bbB}{\mathbb{B}}
\newcommand{\cB}{\mathcal{B}}

\newcommand{\bC}{\mathbf{C}}
\newcommand{\bbC}{\mathbb{C}}

\newcommand{\bbD}{\mathbb{D}}

\newcommand{\bF}{\mathbf{F}}

\newcommand{\bK}{\mathbf{K}}
\newcommand{\cK}{\mathcal{K}}

\newcommand{\bbL}{\mathbb{L}}

\newcommand{\bQ}{\mathbf{Q}}

\newcommand{\bR}{\mathbf{R}}
\newcommand{\bbR}{\mathbb{R}}

\newcommand{\cU}{\mathcal{U}}

\newcommand{\bbZ}{\mathbb{Z}}

\newcommand{\fg}{\mathfrak{g}}

\newcommand{\fp}{\mathfrak{p}}

\newcommand{\ft}{\mathfrak{t}}

\newcommand{\dsum}{\bigoplus}

\newcommand{\GL}{\mathbf{GL}}
\newcommand{\gl}{\mathfrak{gl}}

\newcommand{\Hom}{\operatorname{Hom}}
\newcommand{\Sym}{\operatorname{Sym}}
\newcommand{\Tor}{\operatorname{Tor}}
\newcommand{\Ind}{\operatorname{Ind}}

\newcommand{\DS}{\displaystyle}

\makeatletter
\newcommand{\RN}[1]{%
  \textup{\uppercase\expandafter{\romannumeral#1}}%
  }
\makeatother

\numberwithin{equation}{section}
\newcommand{\arxiv}[1]{\href{http://arxiv.org/abs/#1}{{\tt arXiv:#1}}}

\newdimen\unitsize
\newif\ifnumbered
\newcount\xx
\newcommand\rawpath[1]
{%
\xx=0
\coordinate[vertex] (current) at (0,0);
\foreach\x in {#1}
{
\global\advance\xx by 1
\if\x +%
\def\st{+1}
\else
\if\x 2%
\def\st{+1}
\else
\def\st{-1}
\fi
\fi
\ifnumbered
\draw[edge] (current) -- node[below,green!60!black] {$\scriptstyle \x$} ++(-1-\st,1-\st) coordinate[vertex] (current);
\else
\draw[edge] (current) -- ++(-1-\st,1-\st) coordinate[vertex] (current);
\fi
};
\ifnumbered
\draw[dotted] (0,0) -- (current |- 0,0);
\foreach\x in {1,...,\the\xx}
\path (\x,0)+(-0.5,0) node[below] {\x};
\fi
}

\newcommand{\smcdot}{{\textup{$\cdot$}}}

\title{Syzygies of Determinantal Thickenings}
\author{Hang Huang}
\date{January 2019}

\begin{document}

\begin{abstract}
Let $S = \bbC[x_{i,j}]$ be the ring of polynomial functions on the space of $m \times n$ matrices, and consider the action of the group $\GL = \GL_m \times \GL_n$ via row and column operations on the matrix entries. It is proven in \cite{RW} that for a $\GL$-invariant ideal $I \subseteq S$, the linear strands of its minimal free resolution translates via the BGG correspondence to modules over the general linear Lie superalgebra $\gl(m|n)$. When $I=I_{\lambda}$ is the ideal generated by the $\GL$-orbit of a highest weight vector of weight $\lambda$, they gave a conjectural description of the classes of these $\gl(m|n)$-modules in the Grothendieck group. We prove their conjecture here. We also give a algorithmic description of how to get the classes of these $\gl(m|n)$-modules for any $\GL$-invariant ideal $I \subseteq S$. 
\end{abstract}

\maketitle

\section{Introduction} \label{Introduction}

Consider the vector space $\bbC^{m \times n}$ of $m \times n$ complex matrices ($m \geq n$). Let $S = \bbC[x_{i,j}]$ be its coordinate ring. The group $\GL = \GL_m(\bbC) \times \GL_n(\bbC)$ acts on $\bbC^{m \times n}$ via row and column operations. This makes $S$ into a $\GL$-representation. We can use Cauchy's formula to decompose $S$ into irreducible $\GL$-representations as follows. Write $S_{\lambda}$ for the \emph{Schur functor} associated to a partition $\lambda$ and use \cite[Corollary 2.3.3]{Wey03}, we have that
\begin{align} \label{Cauchy}
\DS S = \dsum_{l(\lambda) \leq n} S_{\lambda} \bbC^m \otimes S_{\lambda} \bbC^n.
\end{align}
When $I \subseteq S$ is a $\GL$-invariant ideal, the \emph{syzygy modules} $\Tor_i^S (I,\bbC)$ are naturally $\GL$-representations. But their explicit description is only known in special cases \cite{Las78} \cite{ABW81} \cite{PW85} \cite{RW17}. 

We can translate the problem of computing syzygies into one modules over the exterior algebra via the BGG correspondence (described in Section ~\ref{subsec:BGG}). In \cite{RW}, Raicu and Weyman have related this to the representation theory of the general linear Lie superalgebra $\gl(m|n)$. In particular, they proved that the linear stands of the minimal free resolution of a $\GL$-invariant ideal translate via the BGG correspondence to finite length $\gl(m|n)$-modules. We consider $I_{\lambda}$ which is generated by a single summand $S_{\lambda} \bbC^m \otimes S_{\lambda} \bbC^n$ in \ref{Cauchy}. This is called a \emph{principal $\GL$-invariant ideal} since it is generated by the $\GL$-orbit of a single highest weight vector. In the same paper \cite{RW}, Raicu and Weyman gave a conjectural description of $\Tor_i^S(I_{\lambda},\bbC)$ for any partition $\lambda$. They have verified their conjecture there for $\lambda = (a^b)$ any rectangular shaped partition by comparing the conjecture with their result in \cite{RW17}. We will prove their conjecture in this paper. We will also give a description of $\Tor_i^S(I,\bbC)$ for any $\GL$-invariant ideal $I \subseteq S$ using the combinatorics of Dyck paths (discussed in Section ~\ref{subsection:partition+Dyck}) and $\gl(m|n)$-representations (discussed in Section ~\ref{subsection:LieSuperalgebra}). The strategy we used here is to put a filtration on the complex of $\gl(m|n)$-modules obtained via the BGG correspondence which encodes the linear strands of the minimal free resolution of $I$. We study the corresponding spectral sequence carefully and prove that it degenerates on the second page in Proposition ~\ref{prop:degenerate}. In most of the proof we will assume $I$ is a principal $\GL$-invariant ideal just for the simplicity of notations. All the arguments remain the same for a general $\GL$-invariant ideal. 

The article is organized as follows. In Section ~\ref{section:Preliminary} we give some background on the combinatorics of partitions and Dyck paths, discuss some basic aspects of the representation theory of general linear Lie algebras and superalgebras, and recall the statement of the BGG correspondence. This section overlaps a lot with Section $2$ in \cite{RW}. In Section ~\ref{section:MainProof} we construct the spectral sequence and prove the main conjecture in \cite{RW}.

\subsection*{Acknowledgements.} 

The author would like to send thanks to Steven Sam for bringing a question which leads to this paper. The author would also like to send special thanks to him for all the fruitful conversations during the writing and conception of this work. The author is also very grateful to Claudiu Raicu and Jerzy Weyman for helpful and inspiring conversations and insightful comments of the early versions of this paper.

\section{Preliminaries and Notations} \label{section:Preliminary}
\subsection{Partitions and Dyck paths.} \label{subsection:partition+Dyck}
Look at a partition $\lambda = (\lambda_1 \geq \lambda_2 \geq \ldots \geq \lambda_n \geq 0)$. We denote $l(\lambda)$ to be the largest $i$ such that $\lambda_i \neq 0$. The size of $\lambda$ is denoted to be $|\lambda| = \lambda_1 + \lambda_2 + \ldots + \lambda_n$. We often omit trailing zeros. For instance, we write $(4,3,1,1)$ for the partition $(4,3,1,1,0,0,0)$. The partial ordering we put on partitions here is that
\begin{align} \label{def:partialordering}
    \lambda \leq \mu  \iff  \lambda_i \leq \mu_i \text{ for all } i=1,2,\ldots,n. 
\end{align}
We will also identify a partition with its associated \emph{Young diagram} as follows. Consider the $2$-dimensional grid induced by the inclusion of $\bbZ^2 \subset \bbR^2$, and index each box in the grid by the coordinates $(x,y)$ of its upper right corner. We identify each partition $\lambda$ with the collection of boxes
\begin{align} \label{def:partition}
    \lambda = \{(i,j) \mid 1 \leq j \leq n, 1 \leq i \leq \lambda_i  \}.
\end{align}
A box $(\lambda_p,p)$ is called a \emph{corner} of the partition $\lambda$ if $\lambda_p > \lambda_{p+1}$. For example, the partition $\lambda = (4,3,1,1)$ has corners $(4,1),(3,2)$ and $(1,4)$ as is pictured as follows:
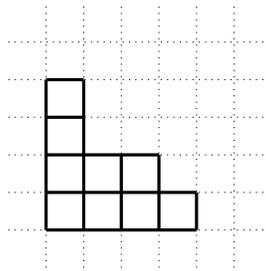
\begin{center}
\begin{tikzpicture}[x=\unitsize,y=\unitsize,baseline=0]
\tikzset{vertex/.style={}}%
\tikzset{edge/.style={very thick}}%
\draw[dotted] (0,0) -- (14,0);
\draw[dotted] (0,2) -- (14,2);
\draw[dotted] (0,4) -- (14,4);
\draw[dotted] (0,6) -- (14,6);
\draw[dotted] (0,8) -- (14,8);
\draw[dotted] (0,10) -- (14,10);
\draw[dotted] (2,-2) -- (2,12);
\draw[dotted] (4,-2) -- (4,12);
\draw[dotted] (6,-2) -- (6,12);
\draw[dotted] (8,-2) -- (8,12);
\draw[dotted] (10,-2) -- (10,12);
\draw[dotted] (12,-2) -- (12,12);
\draw[edge] (2,0) -- (10,0);
\draw[edge] (2,2) -- (10,2);
\draw[edge] (2,4) -- (8,4);
\draw[edge] (2,6) -- (4,6);
\draw[edge] (2,8) -- (4,8);
\draw[edge] (2,0) -- (2,8);
\draw[edge] (4,0) -- (4,8);
\draw[edge] (6,0) -- (6,4);
\draw[edge] (8,0) -- (8,4);
\draw[edge] (10,0) -- (10,2);
\end{tikzpicture}%
\end{center}
If a partition $\lambda = (\lambda_1,\lambda_2,\ldots,\lambda_n)$ has repetitions we use exponential notation. For example, if $\lambda = (3,3,3,1,1,0,0,0)$, then we write $\lambda = (3^3,1^2)$.

A \emph{path} $P$ is a collection of boxes
\begin{align} \label{def:path}
    P = \{ (x_1,y_1),(x_2,y_2),\ldots,(x_k,y_k) \}
\end{align}
satisfying the condition that for each $i = 1,\ldots,k-1$, we have that either
\begin{equation} \label{def:steps}
    (x_{i+1},y_{i+1}) = (x_i + 1,y_i) \text{ or } (x_{i+1},y_{i+1}) = (x_i,y_i-1).
\end{equation}
The \emph{length} of the path $P$ is the number of boxes it contains, namely $k$, and is denoted by $|P|$. A \emph{corner} of $P$ is a box $(x_i,y_i)$ with $1 < i < k$ and $x_{i+1} - x_{i-1} = 1 = y_{i+1} - y_{i-1}$. It is an \emph{inner corner} if $x_{i-1} = x_i$, and an \emph{outer corner} if $y_{i-1} = y_i$. We say that the path $P$ is a \emph{Dyck path of level $d$} if in addition it satisfies
\begin{itemize}
    \item $x_1 + y_1 = x_k + y_k = d$.
    \item $x_i + y_i \geq d$ for all $i = 1,\ldots,k$.
\end{itemize}
Note that the condition $x_1 + y_1 = x_k + y_k = d$ and ~\ref{def:steps} forces the length of a Dyck path to always be an odd number. We illustrate a path $P$ by drawing a broken line segment joining the centers of the squares it contains:

\begin{minipage}{.5\textwidth}
\centering
\begin{tikzpicture}[x=\unitsize,y=\unitsize,baseline=0]
\tikzset{vertex/.style={}}%
\tikzset{edge/.style={very thick}}%
\draw[dotted] (0,0) -- (20,0);
\draw[dotted] (0,2) -- (20,2);
\draw[dotted] (0,4) -- (20,4);
\draw[dotted] (0,6) -- (20,6);
\draw[dotted] (0,8) -- (20,8);
\draw[dotted] (0,10) -- (20,10);
\draw[dotted] (0,12) -- (20,12);
\draw[dotted] (2,-2) -- (2,14);
\draw[dotted] (4,-2) -- (4,14);
\draw[dotted] (6,-2) -- (6,14);
\draw[dotted] (8,-2) -- (8,14);
\draw[dotted] (10,-2) -- (10,14);
\draw[dotted] (12,-2) -- (12,14);
\draw[dotted] (14,-2) -- (14,14);
\draw[dotted] (16,-2) -- (16,14);
\draw[dotted] (18,-2) -- (18,14);
\draw[dotted] (3,13) -- (17,-1);
\draw[red, line width=5pt] (5,11) -- (9,11) -- (9,9) -- (11,9) -- (11,5) -- (13,5) -- (13,3) ;
\end{tikzpicture}
\captionsetup{labelformat=empty,justification=centering}

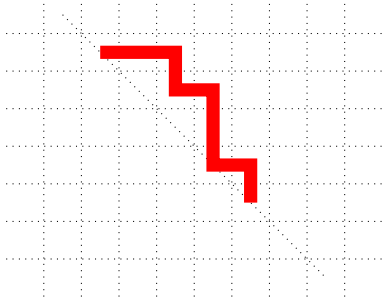
\captionof{figure}{A Dyck path of length $9$, with two inner and three outer corners}
\end{minipage}
\begin{minipage}{.5\textwidth}
\centering
\begin{tikzpicture}[x=\unitsize,y=\unitsize,baseline=0]
\tikzset{vertex/.style={}}%
\tikzset{edge/.style={very thick}}%
\draw[dotted] (0,0) -- (20,0);
\draw[dotted] (0,2) -- (20,2);
\draw[dotted] (0,4) -- (20,4);
\draw[dotted] (0,6) -- (20,6);
\draw[dotted] (0,8) -- (20,8);
\draw[dotted] (0,10) -- (20,10);
\draw[dotted] (0,12) -- (20,12);
\draw[dotted] (2,-2) -- (2,14);
\draw[dotted] (4,-2) -- (4,14);
\draw[dotted] (6,-2) -- (6,14);
\draw[dotted] (8,-2) -- (8,14);
\draw[dotted] (10,-2) -- (10,14);
\draw[dotted] (12,-2) -- (12,14);
\draw[dotted] (14,-2) -- (14,14);
\draw[dotted] (16,-2) -- (16,14);
\draw[dotted] (18,-2) -- (18,14);
\draw[dotted] (3,13) -- (17,-1);
\draw[red, line width=5pt] (5,11) -- (9,11) -- (9,3) -- (15,3) -- (15,1);
\end{tikzpicture}
\captionsetup{labelformat=empty,justification=centering}

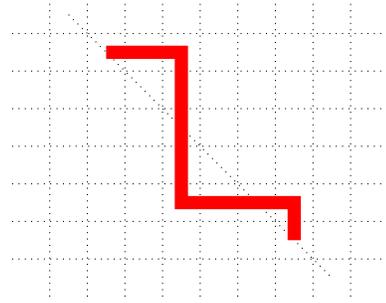
\captionof{figure}{A non-Dyck path of length $11$, with one inner and two outer corners}
\end{minipage}

An \emph{augmented Dyck path} is a pair $\tilde{P} = (P,B)$ where $P$ is a Dyck path and $B$ is a set of boxes, called the \emph{bullets} in $\tilde{P}$, which can be partitioned as $B = B_{\text{head}} \sqcup B_{\text{tail}}$ where (if $P$ is as in \ref{def:path} then)
\begin{equation}
\begin{aligned}
& B_{\text{head}} = \{ (x_1 - a,y_1),(x_1-a+1,y_1),\ldots,(x_1-1,y_1) \} \text{ for some } a \geq 0, \text{ and } \\
& B_{\text{tail}} = \{ (x_k,y_k-1),(x_k,y_k-2),\ldots,(x_k,y_k-b)) \} \text{ for some } b \geq 0.
\end{aligned}
\end{equation}
The length of $\tilde{P}$ is $|\tilde{P}| = |P| + a + b$. To illustrate the augmented Dyck path $\tilde{P}$ we draw $P$ as before, and draw small disks in the center of each of the additional $a+b$ boxes from $B$:
\begin{center}
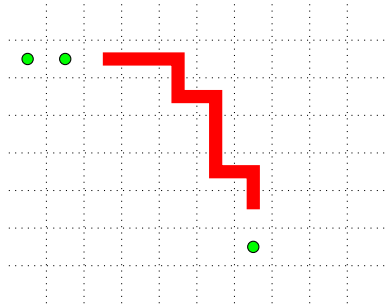

\begin{tikzpicture}[x=\unitsize,y=\unitsize,baseline=0]
\tikzset{vertex/.style={}}%
\tikzset{edge/.style={very thick}}%
\draw[dotted] (0,0) -- (20,0);
\draw[dotted] (0,2) -- (20,2);
\draw[dotted] (0,4) -- (20,4);
\draw[dotted] (0,6) -- (20,6);
\draw[dotted] (0,8) -- (20,8);
\draw[dotted] (0,10) -- (20,10);
\draw[dotted] (0,12) -- (20,12);
\draw[dotted] (2,-2) -- (2,14);
\draw[dotted] (4,-2) -- (4,14);
\draw[dotted] (6,-2) -- (6,14);
\draw[dotted] (8,-2) -- (8,14);
\draw[dotted] (10,-2) -- (10,14);
\draw[dotted] (12,-2) -- (12,14);
\draw[dotted] (14,-2) -- (14,14);
\draw[dotted] (16,-2) -- (16,14);
\draw[dotted] (18,-2) -- (18,14);
\draw[red, line width=5pt] (5,11) -- (9,11) -- (9,9) -- (11,9) -- (11,5) -- (13,5) -- (13,3) ;
\draw[fill=green] (13,1) circle [radius=0.3] ;
\draw[fill=green] (3,11) circle [radius=0.3] ;
\draw[fill=green] (1,11) circle [radius=0.3] ;
\end{tikzpicture}
\captionsetup{labelformat=empty}
\captionof{figure}{An augmented Dyck path of length $12$}
\end{center}

An \emph{(augmented) Dyck pattern} is a collection $\bbD = (D_1,D_2,\ldots,D_r;\bbB)$ where
\begin{itemize}
    \item each $D_i$ is a Dyck path and $\bbB$ is a finite set of boxes,
    \item the sets $D_1,D_2,\ldots,D_r$ and $\bbB$ are pairwise disjoint,
    \item $\bbB$ can be expressed as a union
    \begin{align} \label{def:dyckpattern}
        \bbB = B_1 \cup B_2 \cup \ldots \cup B_r
    \end{align}
    in such a way that $(D_i,B_i)$ is an augmented Dyck path for every $i = 1,2,\ldots,r$.
\end{itemize}

Notice that the collection $(D_1,D_2,\ldots,D_r)$ is invariant under permuting the $D_i$. And we are not requiring the sets $B_i$ in ~\ref{def:dyckpattern} to be disjoint, and in particular we are not asking for the expression ~\ref{def:dyckpattern} to be unique. We also write $\bbD = (D_1,\ldots,D_r)$ when $\bbB = \emptyset$. We define the support of $\bbD$ by
\begin{equation} \label{def:suppD}
    \text{supp}(\bbD) = D_1 \cup D_2 \cup \ldots \cup D_r \cup \bbB.
\end{equation}

If $\lambda$ is a partition and $\bbD$ is a Dyck pattern with supp($\bbD$) disjoint from $\lambda$ (when we think of $\lambda$ as in \ref{def:partition}), then we define
\begin{equation} \label{def:lambdaD}
    \lambda(\bbD) = \lambda \cup \text{supp}(\bbD).
\end{equation}
We say that the Dyck pattern $\bbD$ is \emph{$\lambda$-admissible} if the following conditions are satisfied:
\begin{enumerate}
    \item $\lambda$ is disjoint from supp($\bbD$).
    \item $\lambda(\bbD)$ is (the set of boxes corresponding via ~\ref{def:partition} to) a partition.
    \item For every $i \neq j$, if there exists a box $(x',y') \in D_j$ which is located directly N, E, or NE from a box $(x,y) \in D_i$ (i.e. if $(x',y')$ is one of $(x,y+1),(x+1,y),(x+1,y+1)$), then every box located directly N, E, or NE from a box of $D_i$ must belong to $D_i$ or $D_j$.
    \item There is no bullet in $\bbB$ which is located directly N, E, or NE from a box in any $D_i$.
\end{enumerate}
Below are four examples of $\lambda$-admissible Dyck patterns for $\lambda = (4,3,1,1)$

\begin{minipage}{.24\textwidth}
\centering
\begin{tikzpicture}[x=\unitsize,y=\unitsize,baseline=0]
\tikzset{vertex/.style={}}%
\tikzset{edge/.style={very thick}}%
\draw[dotted] (0,0) -- (14,0);
\draw[dotted] (0,2) -- (14,2);
\draw[dotted] (0,4) -- (14,4);
\draw[dotted] (0,6) -- (14,6);
\draw[dotted] (0,8) -- (14,8);
\draw[dotted] (0,10) -- (14,10);
\draw[dotted] (2,-2) -- (2,12);
\draw[dotted] (4,-2) -- (4,12);
\draw[dotted] (6,-2) -- (6,12);
\draw[dotted] (8,-2) -- (8,12);
\draw[dotted] (10,-2) -- (10,12);
\draw[dotted] (12,-2) -- (12,12);
\draw[edge] (2,0) -- (10,0);
\draw[edge] (2,2) -- (10,2);
\draw[edge] (2,4) -- (8,4);
\draw[edge] (2,6) -- (4,6);
\draw[edge] (2,8) -- (4,8);
\draw[edge] (2,0) -- (2,8);
\draw[edge] (4,0) -- (4,8);
\draw[edge] (6,0) -- (6,4);
\draw[edge] (8,0) -- (8,4);
\draw[edge] (10,0) -- (10,2);
\draw[red, line width=5pt] (5,7) -- (7,7) -- (7,5) ;
\draw[red, line width=5pt] (4.6,5) -- (5.4,5) ;
\draw[red, line width=5pt] (5,9) -- (9,9) -- (9,5) ; 
\draw[fill=green] (9,3) circle [radius=0.3] ; 
\draw[fill=green] (3,9) circle [radius=0.3] ;
\end{tikzpicture}%
\end{minipage}
\begin{minipage}{.24\textwidth}
\centering
\begin{tikzpicture}[x=\unitsize,y=\unitsize,baseline=0]
\tikzset{vertex/.style={}}%
\tikzset{edge/.style={very thick}}%
\draw[dotted] (0,0) -- (14,0);
\draw[dotted] (0,2) -- (14,2);
\draw[dotted] (0,4) -- (14,4);
\draw[dotted] (0,6) -- (14,6);
\draw[dotted] (0,8) -- (14,8);
\draw[dotted] (0,10) -- (14,10);
\draw[dotted] (2,-2) -- (2,12);
\draw[dotted] (4,-2) -- (4,12);
\draw[dotted] (6,-2) -- (6,12);
\draw[dotted] (8,-2) -- (8,12);
\draw[dotted] (10,-2) -- (10,12);
\draw[dotted] (12,-2) -- (12,12);
\draw[edge] (2,0) -- (10,0);
\draw[edge] (2,2) -- (10,2);
\draw[edge] (2,4) -- (8,4);
\draw[edge] (2,6) -- (4,6);
\draw[edge] (2,8) -- (4,8);
\draw[edge] (2,0) -- (2,8);
\draw[edge] (4,0) -- (4,8);
\draw[edge] (6,0) -- (6,4);
\draw[edge] (8,0) -- (8,4);
\draw[edge] (10,0) -- (10,2);
\draw[red, line width=5pt] (5,7) -- (7,7) -- (7,5) ;
\draw[red, line width=5pt] (4.6,5) -- (5.4,5) ;
\draw[red, line width=5pt] (3,9) -- (9,9) -- (9,3) ;
\end{tikzpicture}%
\end{minipage}
\begin{minipage}{.24\textwidth}
\centering
\begin{tikzpicture}[x=\unitsize,y=\unitsize,baseline=0]
\tikzset{vertex/.style={}}%
\tikzset{edge/.style={very thick}}%
\draw[dotted] (0,0) -- (14,0);
\draw[dotted] (0,2) -- (14,2);
\draw[dotted] (0,4) -- (14,4);
\draw[dotted] (0,6) -- (14,6);
\draw[dotted] (0,8) -- (14,8);
\draw[dotted] (0,10) -- (14,10);
\draw[dotted] (2,-2) -- (2,12);
\draw[dotted] (4,-2) -- (4,12);
\draw[dotted] (6,-2) -- (6,12);
\draw[dotted] (8,-2) -- (8,12);
\draw[dotted] (10,-2) -- (10,12);
\draw[dotted] (12,-2) -- (12,12);
\draw[edge] (2,0) -- (10,0);
\draw[edge] (2,2) -- (10,2);
\draw[edge] (2,4) -- (8,4);
\draw[edge] (2,6) -- (4,6);
\draw[edge] (2,8) -- (4,8);
\draw[edge] (2,0) -- (2,8);
\draw[edge] (4,0) -- (4,8);
\draw[edge] (6,0) -- (6,4);
\draw[edge] (8,0) -- (8,4);
\draw[edge] (10,0) -- (10,2);
\draw[red, line width=5pt] (5,7) -- (7,7) -- (7,5) ;
\draw[red, line width=5pt] (4.6,5) -- (5.4,5) ;
\draw[red, line width=5pt] (3,9) -- (9,9) -- (9,3) -- (11,3) -- (11,1);
\end{tikzpicture}%
\end{minipage}
\begin{minipage}{.24\textwidth}
\centering
\begin{tikzpicture}[x=\unitsize,y=\unitsize,baseline=0]
\tikzset{vertex/.style={}}%
\tikzset{edge/.style={very thick}}%
\draw[dotted] (0,0) -- (14,0);
\draw[dotted] (0,2) -- (14,2);
\draw[dotted] (0,4) -- (14,4);
\draw[dotted] (0,6) -- (14,6);
\draw[dotted] (0,8) -- (14,8);
\draw[dotted] (0,10) -- (14,10);
\draw[dotted] (2,-2) -- (2,12);
\draw[dotted] (4,-2) -- (4,12);
\draw[dotted] (6,-2) -- (6,12);
\draw[dotted] (8,-2) -- (8,12);
\draw[dotted] (10,-2) -- (10,12);
\draw[dotted] (12,-2) -- (12,12);
\draw[edge] (2,0) -- (10,0);
\draw[edge] (2,2) -- (10,2);
\draw[edge] (2,4) -- (8,4);
\draw[edge] (2,6) -- (4,6);
\draw[edge] (2,8) -- (4,8);
\draw[edge] (2,0) -- (2,8);
\draw[edge] (4,0) -- (4,8);
\draw[edge] (6,0) -- (6,4);
\draw[edge] (8,0) -- (8,4);
\draw[edge] (10,0) -- (10,2);
\draw[red, line width=5pt] (7,5) -- (9,5) -- (9,3) ;
\draw[red, line width=5pt] (3,9) -- (5,9) -- (5,7) ;
\draw[fill=green] (5,5) circle [radius=0.3] ;
\draw[red, line width=5pt] (10.6,1) -- (11.4,1) ;
\end{tikzpicture}%
\end{minipage}
and three examples of Dyck patterns that are not $\lambda$-admissible

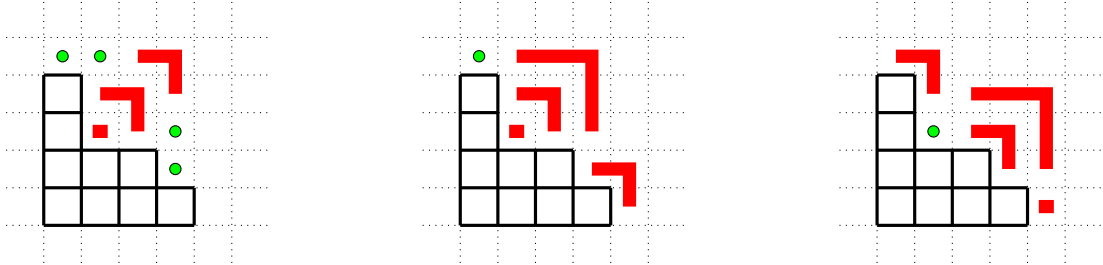
\begin{minipage}{.33\textwidth}
\centering
\begin{tikzpicture}[x=\unitsize,y=\unitsize,baseline=0]
\tikzset{vertex/.style={}}%
\tikzset{edge/.style={very thick}}%
\draw[dotted] (0,0) -- (14,0);
\draw[dotted] (0,2) -- (14,2);
\draw[dotted] (0,4) -- (14,4);
\draw[dotted] (0,6) -- (14,6);
\draw[dotted] (0,8) -- (14,8);
\draw[dotted] (0,10) -- (14,10);
\draw[dotted] (2,-2) -- (2,12);
\draw[dotted] (4,-2) -- (4,12);
\draw[dotted] (6,-2) -- (6,12);
\draw[dotted] (8,-2) -- (8,12);
\draw[dotted] (10,-2) -- (10,12);
\draw[dotted] (12,-2) -- (12,12);
\draw[edge] (2,0) -- (10,0);
\draw[edge] (2,2) -- (10,2);
\draw[edge] (2,4) -- (8,4);
\draw[edge] (2,6) -- (4,6);
\draw[edge] (2,8) -- (4,8);
\draw[edge] (2,0) -- (2,8);
\draw[edge] (4,0) -- (4,8);
\draw[edge] (6,0) -- (6,4);
\draw[edge] (8,0) -- (8,4);
\draw[edge] (10,0) -- (10,2);
\draw[red, line width=5pt] (5,7) -- (7,7) -- (7,5) ;
\draw[red, line width=5pt] (4.6,5) -- (5.4,5) ;
\draw[red, line width=5pt] (7,9) -- (9,9) -- (9,7) ; 
\draw[fill=green] (9,3) circle [radius=0.3] ; 
\draw[fill=green] (3,9) circle [radius=0.3] ;
\draw[fill=green] (9,5) circle [radius=0.3] ; 
\draw[fill=green] (5,9) circle [radius=0.3] ;
\end{tikzpicture}%
\end{minipage}
\begin{minipage}{.33\textwidth}
\centering
\begin{tikzpicture}[x=\unitsize,y=\unitsize,baseline=0]
\tikzset{vertex/.style={}}%
\tikzset{edge/.style={very thick}}%
\draw[dotted] (0,0) -- (14,0);
\draw[dotted] (0,2) -- (14,2);
\draw[dotted] (0,4) -- (14,4);
\draw[dotted] (0,6) -- (14,6);
\draw[dotted] (0,8) -- (14,8);
\draw[dotted] (0,10) -- (14,10);
\draw[dotted] (2,-2) -- (2,12);
\draw[dotted] (4,-2) -- (4,12);
\draw[dotted] (6,-2) -- (6,12);
\draw[dotted] (8,-2) -- (8,12);
\draw[dotted] (10,-2) -- (10,12);
\draw[dotted] (12,-2) -- (12,12);
\draw[edge] (2,0) -- (10,0);
\draw[edge] (2,2) -- (10,2);
\draw[edge] (2,4) -- (8,4);
\draw[edge] (2,6) -- (4,6);
\draw[edge] (2,8) -- (4,8);
\draw[edge] (2,0) -- (2,8);
\draw[edge] (4,0) -- (4,8);
\draw[edge] (6,0) -- (6,4);
\draw[edge] (8,0) -- (8,4);
\draw[edge] (10,0) -- (10,2);
\draw[red, line width=5pt] (5,7) -- (7,7) -- (7,5) ;
\draw[red, line width=5pt] (4.6,5) -- (5.4,5) ;
\draw[fill=green] (3,9) circle [radius=0.3] ;
\draw[red, line width=5pt] (5,9) -- (9,9) -- (9,5) ; 
\draw[red, line width=5pt] (9,3) -- (11,3) -- (11,1) ; 
\end{tikzpicture}%
\end{minipage}
\begin{minipage}{.33\textwidth}
\centering
\begin{tikzpicture}[x=\unitsize,y=\unitsize,baseline=0]
\tikzset{vertex/.style={}}%
\tikzset{edge/.style={very thick}}%
\draw[dotted] (0,0) -- (14,0);
\draw[dotted] (0,2) -- (14,2);
\draw[dotted] (0,4) -- (14,4);
\draw[dotted] (0,6) -- (14,6);
\draw[dotted] (0,8) -- (14,8);
\draw[dotted] (0,10) -- (14,10);
\draw[dotted] (2,-2) -- (2,12);
\draw[dotted] (4,-2) -- (4,12);
\draw[dotted] (6,-2) -- (6,12);
\draw[dotted] (8,-2) -- (8,12);
\draw[dotted] (10,-2) -- (10,12);
\draw[dotted] (12,-2) -- (12,12);
\draw[edge] (2,0) -- (10,0);
\draw[edge] (2,2) -- (10,2);
\draw[edge] (2,4) -- (8,4);
\draw[edge] (2,6) -- (4,6);
\draw[edge] (2,8) -- (4,8);
\draw[edge] (2,0) -- (2,8);
\draw[edge] (4,0) -- (4,8);
\draw[edge] (6,0) -- (6,4);
\draw[edge] (8,0) -- (8,4);
\draw[edge] (10,0) -- (10,2);
\draw[red, line width=5pt] (3,9) -- (5,9) -- (5,7) ;
\draw[red, line width=5pt] (7,7) -- (11,7) -- (11,3) ;
\draw[red, line width=5pt] (7,5) -- (9,5) -- (9,3) ;
\draw[fill=green] (5,5) circle [radius=0.3] ;
\draw[red, line width=5pt] (10.6,1) -- (11.4,1) ;
\end{tikzpicture}%
\end{minipage}

For a fixed $\lambda$, let $\bbD = (D_1,\ldots,D_r ; \bbB)$ be a $\lambda$-admissible Dyck pattern. Define
\begin{align} \label{def:lambdaB}
    \lambda(\bbB) = \lambda \sqcup \bbB.
\end{align}
We have the following two facts:
\begin{enumerate}
    \item The last condition in the definition of $\lambda$-admissible Dyck patterns is equivalent to the fact that $\lambda(\bbB)$ is (the set of boxes corresponding via ~\ref{def:partition} to) a partition.
    \item If $(D_1,D_2,\ldots,D_k)$ is a subcollection of $(D_1,D_2,\ldots,D_r)$ such that $|D_i| = 1$ for all $i = 1,2,\ldots,k$. And for any $i=1,2,\ldots,k$, the boxes located directed W and S of $D_i$ are not in $\bbB$. If. $D_i = \{(x,y)\}$, we require $(x-1,y),(x,y-1) \notin \bbB$. Then $\bbD'$ is a $\mu$-admissible Dyck pattern where $\bbD' = (D_{k+1},D_{k+2},\ldots,D_{r};\bbB)$ and $\mu = \lambda(\bbD_0)$ where $\bbD_0 = (D_1,D_2,\ldots,D_k)$.
\end{enumerate}

We also define the \emph{Dyck size} of $\bbD$ to be
\begin{align} \label{def:Dycksize}
    d(\bbD) = |D_1| + |D_2| + \ldots + |D_r|
\end{align}
and the \emph{bullet size} of $\bbD$ to be
\begin{align} \label{def:Bulletsize}
    b(\bbD) = |\bbB|.
\end{align}
The \emph{(total) size} of $\bbD$ is $|\bbD| = d(\bbD) + b(\bbD)$, so that $|\lambda(\bbD)| = |\lambda| + |\bbD|$ for every $\lambda$-admissible Dyck pattern $\bbD$.

\subsection{The general linear Lie algebra.} \label{subsection:LieAlgebra}
We use $U$ to denote a finite dimensional complex vector space with $\dim (U) = n$. Let $\gl(U)$ be the Lie algebra of endomorphisms of $U$, with the usual Lie bracket $[x,y] = xy-yx$. For any partition $\lambda = (\lambda_1,\lambda_2,\ldots,\lambda_n)$, we write $S_{\lambda}$ for the \emph{Schur functor} associated to $\lambda$. We also denote the dual vector space to be $U^{\vee} = \Hom_{\bbC}(U,\bbC)$. And $\lambda^{\vee} = (-\lambda_n,-\lambda_{n-1},\ldots,-\lambda_1)$. We have a natural isomorphism
\begin{equation} \label{SchurFunctorDualVectorSpace}
    S_{\lambda} U^{\vee} \simeq S_{\lambda^{\vee}} U.
\end{equation}
The convention for Schur functors we chose here is so that if $\lambda = (d)$ for $d \geq 0$, then $S_{\lambda} U = \Sym^d U$. And if $\lambda = (1^k)$, then $S_{\lambda} U = \bigwedge^k U$.

A choice of basis on $U$ determines a maximal torus $\ft$ of diagonal matrices inside $\gl(U)$, and a dual basis of $U^{\vee}$ with a corresponding maximal torus $\ft^{\vee}$ inside $\gl(U^{\vee})$. There is a natural identification $\gl(U) \simeq \gl(U^{\vee})$ that sends $\phi \mapsto - \phi^{\vee}$. Through this identification, positive weights with respect to $\ft$ will correspond to negative weights with respect to $\ft^{\vee}$ and vice-versa. Here we choose our conventions so that we only need to work with partitions $\lambda$ in the rest of the article.

\subsection{Representations of the general linear Lie superalgebra.} \label{subsection:LieSuperalgebra}
Throughout this article we let $V_0,V_1$ be complex vector spaces with $\dim(V_0) = m$ and $\dim(V_1) = n$, and assume that $m \geq n$. We write $W_i = V_i^{\vee}$ for their vector space duals. And let $V = V_0 \otimes V_1$ and $W = W_0 \otimes W_1 = V^{\vee}$. Consider the polynomial ring $S = \Sym(V)$ and the exterior algebra $E = \bigwedge W$. Choosing dual basis on the spaces $V_i$ and $W_i$, we can identify $S = \bbC[x_{i,j}]$ and $E = \bbC \langle e_{i,j} \rangle$ where $\langle , \rangle$ indicates that the multiplication in $E$ is skew-commutative.

We use $\fg = \gl(m|n)$ to denote the general linear Lie superalgebra of endomorphisms of the $\bbZ / 2 \bbZ$-graded vector space $W_0 \oplus V_1$ where $W_0 \simeq \bbC^m$ lies in degree $0$, and $V_1 \simeq \bbC^n$ lies in degree $1$. Consider the $\bbZ$-grading on $\fg$ given by
\begin{align*}
    & \fg_0 = \gl(V_0) \oplus \gl(W_1) \simeq \gl(W_0) \oplus \gl(W_1), \\
    & \fg_{-1} = \Hom_{\bbC}(V_0,W_1) \simeq W_0 \otimes W_1, \\
    & \fg_1 = \Hom_{\bbC}(W_1,V_0) \simeq V_0 \otimes V_1.
\end{align*}
The Lie superbracket is given by $[x,y] = xy - (-1)^{\deg(x)\deg(y)}yx$ for $x,y$ homogeneous elements of $\fg$. Note that the Lie superbracket restricts to the usual Lie bracket on $\fg_0$, which is a reductive Lie algebra. We define
\[
\fp = \fg_0 \oplus \fg_1
\]
which is a parabolic subalgebra of $\fg$. Every $\fg_0$ module $M$ can be thought of as a $\fp$-module by making $\fg_1$ to act on $M$ trivially. For a partition $\lambda = (\lambda_1,\lambda_2,\ldots,\lambda_n)$, consider the irreducible $\fg_0$-module $S_{\lambda} W_0 \otimes S_{\lambda} W_1$. We can think of it as a $\fp$-module, and define the induced representation to be
\begin{equation} \label{def:kac}
    K_{\lambda} = \Ind_{\fp}^{\fg}(S_{\lambda} W_0 \otimes S_{\lambda} W_1)
\end{equation}
This is called the \emph{Kac module of weight $\lambda$}. In general we consider a more general version of Kac modules by inducing from $S_{\lambda} W_0 \otimes S_{\mu} W_1 $ for an arbitrary pair of partitions $(\lambda,\mu)$. The special case of Kac modules that we consider here in \ref{def:kac} are the ones of \emph{maximal degree of atypicality}. They are all of degree of atypicality equals $\min (m,n)$. For a detail discussion about atypicality, see \cite{CW}[Section 2.2.6]. 

\begin{defn} \label{defn:Loewy}
A \emph{Loewy filtration} on a $\fg$-module $V$ is a filtration such that all consecutive quotients of it are semisimple $\fg$-modules.
\end{defn}
 In general this filtration might not be unique. A good feature about general Kac modules $K$ of $\gl(m|n)$ is that all of them have a unique Loewy filtration of length the degree of atypicality of $K$ \cite{SZ}[Theorem 3.2]. We will make use of this filtration in the main proof.

Note that $\fg_{-1} = W_0 \otimes W_1$ is an abelian Lie superalgebra concentrated in odd degree. So its universal enveloping algebra is $\cU(\fg_{-1}) = \bigwedge \fg_{-1} = \bigwedge W = E$. In fact, as a $E$-module, the Kac module of weight $\lambda$ is
\begin{equation}
    K_{\lambda} = E \otimes (S_{\lambda} W_0 \otimes S_{\lambda} W_1).
\end{equation}
The $\fg_0$-module structure of $K_{\lambda}$ can be obtained based on the Cauchy decomposition of exterior power of a tensor product, combined with the Littlewood-Richardson rule. As a $\fg$-module, $K_{\lambda}$ has a unique simple quotient, which we denote by $\bbL_{\lambda}$. Here $\bbL_{\lambda}$ is the simple $\fg$-module of highest weight $\lambda$. As a $\fg$-module, $K_{\lambda}$ is not semi-simple. But it has finite length with composition factors described as follows. Let
\begin{equation} \label{def:Klambdan}
    \cK(\lambda;n) = \{ (\bbD) = (D_1,\ldots,D_r) \mbox{ a $\lambda$-admissible Dyck pattern} \mid l(\lambda(\bbD)) \leq n \}.
\end{equation}
We stress the fact that the Dyck patterns in $\cK(\lambda;n)$ are not augmented, i.e., they contain no bullets. But they may contain Dyck paths of length one. The composition factors of the Kac modules are encoded by parabolic version of Kazhdan-Lusztig polynomials \cite{Bru03} \cite{Ser96}. Using the Dyck pattern interpretation of the parabolic Kaszhdan-Lusztig polynomials based on Rule II in \cite{SZJ12}[Section 3.1], we get the following theorem.

\begin{thm} \label{thm:CompositionSeriesKacModule}
If we let $[M]$ denote the class of a $\fg$-module $M$ in the Grothendieck group $K_0(\fg)$ of finite dimensional representations of $\fg$, then
\begin{equation}
\DS    [K_{\lambda}] = \sum_{\bbD \in \cK(\lambda;n)} [\bbL_{\lambda(\bbD)}]. 
\end{equation}
\end{thm}

\begin{exmp}
Take $n=4$ and consider $\lambda = (4,3,1,1)$. Take the Kac module $K_{\lambda}$ has $19$ simple composition factors, classified by the Dyck patterns $\bbD$ pictured below (and labelled by $\lambda(\bbD)$)

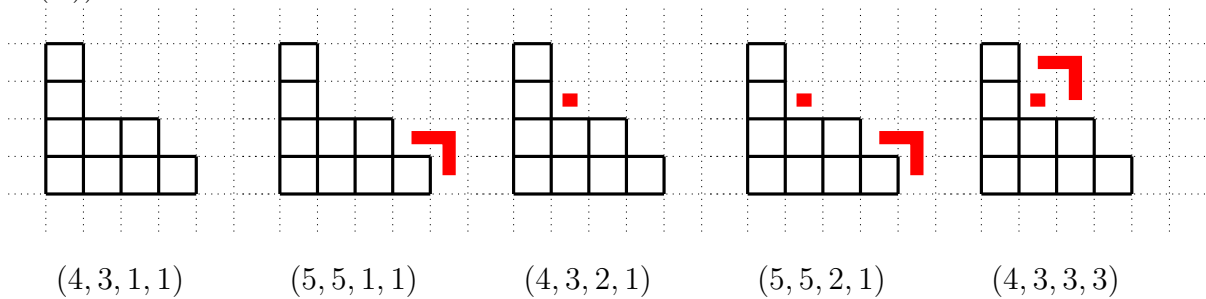
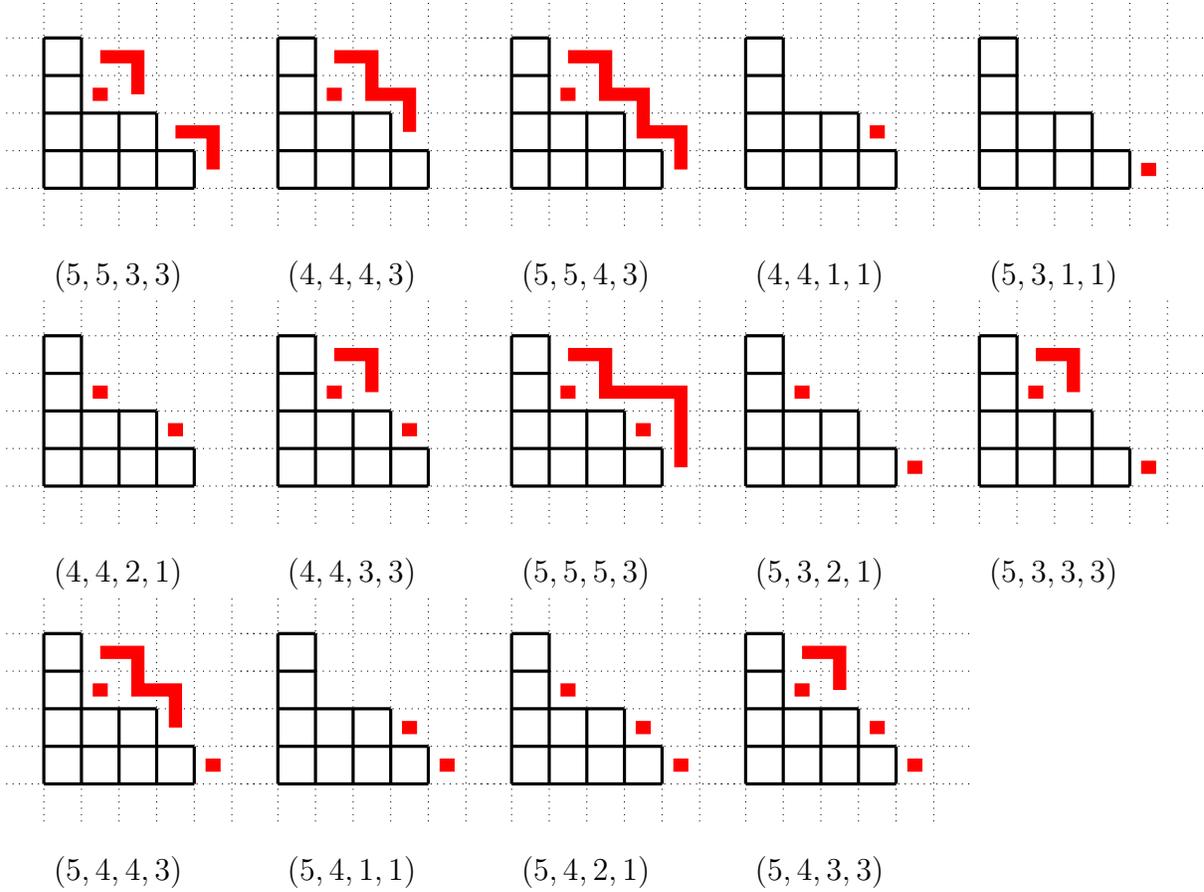
\begin{minipage}{.18\textwidth}
\centering
\begin{tikzpicture}[x=\unitsize,y=\unitsize,baseline=0]
\tikzset{vertex/.style={}}%
\tikzset{edge/.style={very thick}}%
\draw[dotted] (0,0) -- (14,0);
\draw[dotted] (0,2) -- (14,2);
\draw[dotted] (0,4) -- (14,4);
\draw[dotted] (0,6) -- (14,6);
\draw[dotted] (0,8) -- (14,8);
\draw[dotted] (2,-2) -- (2,10);
\draw[dotted] (4,-2) -- (4,10);
\draw[dotted] (6,-2) -- (6,10);
\draw[dotted] (8,-2) -- (8,10);
\draw[dotted] (10,-2) -- (10,10);
\draw[dotted] (12,-2) -- (12,10);
\draw[edge] (2,0) -- (10,0);
\draw[edge] (2,2) -- (10,2);
\draw[edge] (2,4) -- (8,4);
\draw[edge] (2,6) -- (4,6);
\draw[edge] (2,8) -- (4,8);
\draw[edge] (2,0) -- (2,8);
\draw[edge] (4,0) -- (4,8);
\draw[edge] (6,0) -- (6,4);
\draw[edge] (8,0) -- (8,4);
\draw[edge] (10,0) -- (10,2);
\end{tikzpicture}%
\captionsetup{labelformat=empty}
\captionof{figure}{$(4,3,1,1)$}
\end{minipage}
\begin{minipage}{.18\textwidth}
\centering
\begin{tikzpicture}[x=\unitsize,y=\unitsize,baseline=0]
\tikzset{vertex/.style={}}%
\tikzset{edge/.style={very thick}}%
\draw[dotted] (0,0) -- (14,0);
\draw[dotted] (0,2) -- (14,2);
\draw[dotted] (0,4) -- (14,4);
\draw[dotted] (0,6) -- (14,6);
\draw[dotted] (0,8) -- (14,8);
\draw[dotted] (2,-2) -- (2,10);
\draw[dotted] (4,-2) -- (4,10);
\draw[dotted] (6,-2) -- (6,10);
\draw[dotted] (8,-2) -- (8,10);
\draw[dotted] (10,-2) -- (10,10);
\draw[dotted] (12,-2) -- (12,10);
\draw[edge] (2,0) -- (10,0);
\draw[edge] (2,2) -- (10,2);
\draw[edge] (2,4) -- (8,4);
\draw[edge] (2,6) -- (4,6);
\draw[edge] (2,8) -- (4,8);
\draw[edge] (2,0) -- (2,8);
\draw[edge] (4,0) -- (4,8);
\draw[edge] (6,0) -- (6,4);
\draw[edge] (8,0) -- (8,4);
\draw[edge] (10,0) -- (10,2);
\draw[red, line width=5pt] (9,3) -- (11,3) -- (11,1) ;
\end{tikzpicture}%
\captionsetup{labelformat=empty}
\captionof{figure}{$(5,5,1,1)$}
\end{minipage}
\begin{minipage}{.18\textwidth}
\centering
\begin{tikzpicture}[x=\unitsize,y=\unitsize,baseline=0]
\tikzset{vertex/.style={}}%
\tikzset{edge/.style={very thick}}%
\draw[dotted] (0,0) -- (14,0);
\draw[dotted] (0,2) -- (14,2);
\draw[dotted] (0,4) -- (14,4);
\draw[dotted] (0,6) -- (14,6);
\draw[dotted] (0,8) -- (14,8);
\draw[dotted] (2,-2) -- (2,10);
\draw[dotted] (4,-2) -- (4,10);
\draw[dotted] (6,-2) -- (6,10);
\draw[dotted] (8,-2) -- (8,10);
\draw[dotted] (10,-2) -- (10,10);
\draw[dotted] (12,-2) -- (12,10);
\draw[edge] (2,0) -- (10,0);
\draw[edge] (2,2) -- (10,2);
\draw[edge] (2,4) -- (8,4);
\draw[edge] (2,6) -- (4,6);
\draw[edge] (2,8) -- (4,8);
\draw[edge] (2,0) -- (2,8);
\draw[edge] (4,0) -- (4,8);
\draw[edge] (6,0) -- (6,4);
\draw[edge] (8,0) -- (8,4);
\draw[edge] (10,0) -- (10,2);
\draw[red, line width=5pt] (4.6,5) -- (5.4,5) ;
\end{tikzpicture}%
\captionsetup{labelformat=empty}
\captionof{figure}{$(4,3,2,1)$}
\end{minipage}
\begin{minipage}{.18\textwidth}
\centering
\begin{tikzpicture}[x=\unitsize,y=\unitsize,baseline=0]
\tikzset{vertex/.style={}}%
\tikzset{edge/.style={very thick}}%
\draw[dotted] (0,0) -- (14,0);
\draw[dotted] (0,2) -- (14,2);
\draw[dotted] (0,4) -- (14,4);
\draw[dotted] (0,6) -- (14,6);
\draw[dotted] (0,8) -- (14,8);
\draw[dotted] (2,-2) -- (2,10);
\draw[dotted] (4,-2) -- (4,10);
\draw[dotted] (6,-2) -- (6,10);
\draw[dotted] (8,-2) -- (8,10);
\draw[dotted] (10,-2) -- (10,10);
\draw[dotted] (12,-2) -- (12,10);
\draw[edge] (2,0) -- (10,0);
\draw[edge] (2,2) -- (10,2);
\draw[edge] (2,4) -- (8,4);
\draw[edge] (2,6) -- (4,6);
\draw[edge] (2,8) -- (4,8);
\draw[edge] (2,0) -- (2,8);
\draw[edge] (4,0) -- (4,8);
\draw[edge] (6,0) -- (6,4);
\draw[edge] (8,0) -- (8,4);
\draw[edge] (10,0) -- (10,2);
\draw[red, line width=5pt] (9,3) -- (11,3) -- (11,1) ;
\draw[red, line width=5pt] (4.6,5) -- (5.4,5) ;
\end{tikzpicture}%
\captionsetup{labelformat=empty}
\captionof{figure}{$(5,5,2,1)$}
\end{minipage}
\begin{minipage}{.18\textwidth}
\centering
\begin{tikzpicture}[x=\unitsize,y=\unitsize,baseline=0]
\tikzset{vertex/.style={}}%
\tikzset{edge/.style={very thick}}%
\draw[dotted] (0,0) -- (14,0);
\draw[dotted] (0,2) -- (14,2);
\draw[dotted] (0,4) -- (14,4);
\draw[dotted] (0,6) -- (14,6);
\draw[dotted] (0,8) -- (14,8);
\draw[dotted] (2,-2) -- (2,10);
\draw[dotted] (4,-2) -- (4,10);
\draw[dotted] (6,-2) -- (6,10);
\draw[dotted] (8,-2) -- (8,10);
\draw[dotted] (10,-2) -- (10,10);
\draw[dotted] (12,-2) -- (12,10);
\draw[edge] (2,0) -- (10,0);
\draw[edge] (2,2) -- (10,2);
\draw[edge] (2,4) -- (8,4);
\draw[edge] (2,6) -- (4,6);
\draw[edge] (2,8) -- (4,8);
\draw[edge] (2,0) -- (2,8);
\draw[edge] (4,0) -- (4,8);
\draw[edge] (6,0) -- (6,4);
\draw[edge] (8,0) -- (8,4);
\draw[edge] (10,0) -- (10,2);
\draw[red, line width=5pt] (5,7) -- (7,7) -- (7,5) ;
\draw[red, line width=5pt] (4.6,5) -- (5.4,5) ;
\end{tikzpicture}%
\captionsetup{labelformat=empty}
\captionof{figure}{$(4,3,3,3)$}
\end{minipage}

\begin{minipage}{.18\textwidth}
\centering
\begin{tikzpicture}[x=\unitsize,y=\unitsize,baseline=0]
\tikzset{vertex/.style={}}%
\tikzset{edge/.style={very thick}}%
\draw[dotted] (0,0) -- (14,0);
\draw[dotted] (0,2) -- (14,2);
\draw[dotted] (0,4) -- (14,4);
\draw[dotted] (0,6) -- (14,6);
\draw[dotted] (0,8) -- (14,8);
\draw[dotted] (2,-2) -- (2,10);
\draw[dotted] (4,-2) -- (4,10);
\draw[dotted] (6,-2) -- (6,10);
\draw[dotted] (8,-2) -- (8,10);
\draw[dotted] (10,-2) -- (10,10);
\draw[dotted] (12,-2) -- (12,10);
\draw[edge] (2,0) -- (10,0);
\draw[edge] (2,2) -- (10,2);
\draw[edge] (2,4) -- (8,4);
\draw[edge] (2,6) -- (4,6);
\draw[edge] (2,8) -- (4,8);
\draw[edge] (2,0) -- (2,8);
\draw[edge] (4,0) -- (4,8);
\draw[edge] (6,0) -- (6,4);
\draw[edge] (8,0) -- (8,4);
\draw[edge] (10,0) -- (10,2);
\draw[red, line width=5pt] (5,7) -- (7,7) -- (7,5) ;
\draw[red, line width=5pt] (4.6,5) -- (5.4,5) ;
\draw[red, line width=5pt] (9,3) -- (11,3) -- (11,1) ;
\end{tikzpicture}%
\captionsetup{labelformat=empty}
\captionof{figure}{$(5,5,3,3)$}
\end{minipage}
\begin{minipage}{.18\textwidth}
\centering
\begin{tikzpicture}[x=\unitsize,y=\unitsize,baseline=0]
\tikzset{vertex/.style={}}%
\tikzset{edge/.style={very thick}}%
\draw[dotted] (0,0) -- (14,0);
\draw[dotted] (0,2) -- (14,2);
\draw[dotted] (0,4) -- (14,4);
\draw[dotted] (0,6) -- (14,6);
\draw[dotted] (0,8) -- (14,8);
\draw[dotted] (2,-2) -- (2,10);
\draw[dotted] (4,-2) -- (4,10);
\draw[dotted] (6,-2) -- (6,10);
\draw[dotted] (8,-2) -- (8,10);
\draw[dotted] (10,-2) -- (10,10);
\draw[dotted] (12,-2) -- (12,10);
\draw[edge] (2,0) -- (10,0);
\draw[edge] (2,2) -- (10,2);
\draw[edge] (2,4) -- (8,4);
\draw[edge] (2,6) -- (4,6);
\draw[edge] (2,8) -- (4,8);
\draw[edge] (2,0) -- (2,8);
\draw[edge] (4,0) -- (4,8);
\draw[edge] (6,0) -- (6,4);
\draw[edge] (8,0) -- (8,4);
\draw[edge] (10,0) -- (10,2);
\draw[red, line width=5pt] (5,7) -- (7,7) -- (7,5) -- (9,5) -- (9,3);
\draw[red, line width=5pt] (4.6,5) -- (5.4,5) ;
\end{tikzpicture}%
\captionsetup{labelformat=empty}
\captionof{figure}{$(4,4,4,3)$}
\end{minipage}
\begin{minipage}{.18\textwidth}
\centering
\begin{tikzpicture}[x=\unitsize,y=\unitsize,baseline=0]
\tikzset{vertex/.style={}}%
\tikzset{edge/.style={very thick}}%
\draw[dotted] (0,0) -- (14,0);
\draw[dotted] (0,2) -- (14,2);
\draw[dotted] (0,4) -- (14,4);
\draw[dotted] (0,6) -- (14,6);
\draw[dotted] (0,8) -- (14,8);
\draw[dotted] (2,-2) -- (2,10);
\draw[dotted] (4,-2) -- (4,10);
\draw[dotted] (6,-2) -- (6,10);
\draw[dotted] (8,-2) -- (8,10);
\draw[dotted] (10,-2) -- (10,10);
\draw[dotted] (12,-2) -- (12,10);
\draw[edge] (2,0) -- (10,0);
\draw[edge] (2,2) -- (10,2);
\draw[edge] (2,4) -- (8,4);
\draw[edge] (2,6) -- (4,6);
\draw[edge] (2,8) -- (4,8);
\draw[edge] (2,0) -- (2,8);
\draw[edge] (4,0) -- (4,8);
\draw[edge] (6,0) -- (6,4);
\draw[edge] (8,0) -- (8,4);
\draw[edge] (10,0) -- (10,2);
\draw[red, line width=5pt] (5,7) -- (7,7) -- (7,5) -- (9,5) -- (9,3) -- (11,3) -- (11,1);
\draw[red, line width=5pt] (4.6,5) -- (5.4,5) ;
\end{tikzpicture}%
\captionsetup{labelformat=empty}
\captionof{figure}{$(5,5,4,3)$}
\end{minipage}
\begin{minipage}{.18\textwidth}
\centering
\begin{tikzpicture}[x=\unitsize,y=\unitsize,baseline=0]
\tikzset{vertex/.style={}}%
\tikzset{edge/.style={very thick}}%
\draw[dotted] (0,0) -- (14,0);
\draw[dotted] (0,2) -- (14,2);
\draw[dotted] (0,4) -- (14,4);
\draw[dotted] (0,6) -- (14,6);
\draw[dotted] (0,8) -- (14,8);
\draw[dotted] (2,-2) -- (2,10);
\draw[dotted] (4,-2) -- (4,10);
\draw[dotted] (6,-2) -- (6,10);
\draw[dotted] (8,-2) -- (8,10);
\draw[dotted] (10,-2) -- (10,10);
\draw[dotted] (12,-2) -- (12,10);
\draw[edge] (2,0) -- (10,0);
\draw[edge] (2,2) -- (10,2);
\draw[edge] (2,4) -- (8,4);
\draw[edge] (2,6) -- (4,6);
\draw[edge] (2,8) -- (4,8);
\draw[edge] (2,0) -- (2,8);
\draw[edge] (4,0) -- (4,8);
\draw[edge] (6,0) -- (6,4);
\draw[edge] (8,0) -- (8,4);
\draw[edge] (10,0) -- (10,2);
\draw[red, line width=5pt] (8.6,3) -- (9.4,3) ;
\end{tikzpicture}%
\captionsetup{labelformat=empty}
\captionof{figure}{$(4,4,1,1)$}
\end{minipage}
\begin{minipage}{.18\textwidth}
\centering
\begin{tikzpicture}[x=\unitsize,y=\unitsize,baseline=0]
\tikzset{vertex/.style={}}%
\tikzset{edge/.style={very thick}}%
\draw[dotted] (0,0) -- (14,0);
\draw[dotted] (0,2) -- (14,2);
\draw[dotted] (0,4) -- (14,4);
\draw[dotted] (0,6) -- (14,6);
\draw[dotted] (0,8) -- (14,8);
\draw[dotted] (2,-2) -- (2,10);
\draw[dotted] (4,-2) -- (4,10);
\draw[dotted] (6,-2) -- (6,10);
\draw[dotted] (8,-2) -- (8,10);
\draw[dotted] (10,-2) -- (10,10);
\draw[dotted] (12,-2) -- (12,10);
\draw[edge] (2,0) -- (10,0);
\draw[edge] (2,2) -- (10,2);
\draw[edge] (2,4) -- (8,4);
\draw[edge] (2,6) -- (4,6);
\draw[edge] (2,8) -- (4,8);
\draw[edge] (2,0) -- (2,8);
\draw[edge] (4,0) -- (4,8);
\draw[edge] (6,0) -- (6,4);
\draw[edge] (8,0) -- (8,4);
\draw[edge] (10,0) -- (10,2);
\draw[red, line width=5pt] (10.6,1) -- (11.4,1) ;
\end{tikzpicture}%
\captionsetup{labelformat=empty}
\captionof{figure}{$(5,3,1,1)$}
\end{minipage}

\begin{minipage}{.18\textwidth}
\centering
\begin{tikzpicture}[x=\unitsize,y=\unitsize,baseline=0]
\tikzset{vertex/.style={}}%
\tikzset{edge/.style={very thick}}%
\draw[dotted] (0,0) -- (14,0);
\draw[dotted] (0,2) -- (14,2);
\draw[dotted] (0,4) -- (14,4);
\draw[dotted] (0,6) -- (14,6);
\draw[dotted] (0,8) -- (14,8);
\draw[dotted] (2,-2) -- (2,10);
\draw[dotted] (4,-2) -- (4,10);
\draw[dotted] (6,-2) -- (6,10);
\draw[dotted] (8,-2) -- (8,10);
\draw[dotted] (10,-2) -- (10,10);
\draw[dotted] (12,-2) -- (12,10);
\draw[edge] (2,0) -- (10,0);
\draw[edge] (2,2) -- (10,2);
\draw[edge] (2,4) -- (8,4);
\draw[edge] (2,6) -- (4,6);
\draw[edge] (2,8) -- (4,8);
\draw[edge] (2,0) -- (2,8);
\draw[edge] (4,0) -- (4,8);
\draw[edge] (6,0) -- (6,4);
\draw[edge] (8,0) -- (8,4);
\draw[edge] (10,0) -- (10,2);
\draw[red, line width=5pt] (8.6,3) -- (9.4,3) ;
\draw[red, line width=5pt] (4.6,5) -- (5.4,5) ;
\end{tikzpicture}%
\captionsetup{labelformat=empty}
\captionof{figure}{$(4,4,2,1)$}
\end{minipage}
\begin{minipage}{.18\textwidth}
\centering
\begin{tikzpicture}[x=\unitsize,y=\unitsize,baseline=0]
\tikzset{vertex/.style={}}%
\tikzset{edge/.style={very thick}}%
\draw[dotted] (0,0) -- (14,0);
\draw[dotted] (0,2) -- (14,2);
\draw[dotted] (0,4) -- (14,4);
\draw[dotted] (0,6) -- (14,6);
\draw[dotted] (0,8) -- (14,8);
\draw[dotted] (2,-2) -- (2,10);
\draw[dotted] (4,-2) -- (4,10);
\draw[dotted] (6,-2) -- (6,10);
\draw[dotted] (8,-2) -- (8,10);
\draw[dotted] (10,-2) -- (10,10);
\draw[dotted] (12,-2) -- (12,10);
\draw[edge] (2,0) -- (10,0);
\draw[edge] (2,2) -- (10,2);
\draw[edge] (2,4) -- (8,4);
\draw[edge] (2,6) -- (4,6);
\draw[edge] (2,8) -- (4,8);
\draw[edge] (2,0) -- (2,8);
\draw[edge] (4,0) -- (4,8);
\draw[edge] (6,0) -- (6,4);
\draw[edge] (8,0) -- (8,4);
\draw[edge] (10,0) -- (10,2);
\draw[red, line width=5pt] (8.6,3) -- (9.4,3) ;
\draw[red, line width=5pt] (4.6,5) -- (5.4,5) ;
\draw[red, line width=5pt] (5,7) -- (7,7) -- (7,5) ;
\end{tikzpicture}%
\captionsetup{labelformat=empty}
\captionof{figure}{$(4,4,3,3)$}
\end{minipage}
\begin{minipage}{.18\textwidth}
\centering
\begin{tikzpicture}[x=\unitsize,y=\unitsize,baseline=0]
\tikzset{vertex/.style={}}%
\tikzset{edge/.style={very thick}}%
\draw[dotted] (0,0) -- (14,0);
\draw[dotted] (0,2) -- (14,2);
\draw[dotted] (0,4) -- (14,4);
\draw[dotted] (0,6) -- (14,6);
\draw[dotted] (0,8) -- (14,8);
\draw[dotted] (2,-2) -- (2,10);
\draw[dotted] (4,-2) -- (4,10);
\draw[dotted] (6,-2) -- (6,10);
\draw[dotted] (8,-2) -- (8,10);
\draw[dotted] (10,-2) -- (10,10);
\draw[dotted] (12,-2) -- (12,10);
\draw[edge] (2,0) -- (10,0);
\draw[edge] (2,2) -- (10,2);
\draw[edge] (2,4) -- (8,4);
\draw[edge] (2,6) -- (4,6);
\draw[edge] (2,8) -- (4,8);
\draw[edge] (2,0) -- (2,8);
\draw[edge] (4,0) -- (4,8);
\draw[edge] (6,0) -- (6,4);
\draw[edge] (8,0) -- (8,4);
\draw[edge] (10,0) -- (10,2);
\draw[red, line width=5pt] (8.6,3) -- (9.4,3) ;
\draw[red, line width=5pt] (4.6,5) -- (5.4,5) ;
\draw[red, line width=5pt] (5,7) -- (7,7) -- (7,5) -- (11,5) -- (11,1);
\end{tikzpicture}%
\captionsetup{labelformat=empty}
\captionof{figure}{$(5,5,5,3)$}
\end{minipage}
\begin{minipage}{.18\textwidth}
\centering
\begin{tikzpicture}[x=\unitsize,y=\unitsize,baseline=0]
\tikzset{vertex/.style={}}%
\tikzset{edge/.style={very thick}}%
\draw[dotted] (0,0) -- (14,0);
\draw[dotted] (0,2) -- (14,2);
\draw[dotted] (0,4) -- (14,4);
\draw[dotted] (0,6) -- (14,6);
\draw[dotted] (0,8) -- (14,8);
\draw[dotted] (2,-2) -- (2,10);
\draw[dotted] (4,-2) -- (4,10);
\draw[dotted] (6,-2) -- (6,10);
\draw[dotted] (8,-2) -- (8,10);
\draw[dotted] (10,-2) -- (10,10);
\draw[dotted] (12,-2) -- (12,10);
\draw[edge] (2,0) -- (10,0);
\draw[edge] (2,2) -- (10,2);
\draw[edge] (2,4) -- (8,4);
\draw[edge] (2,6) -- (4,6);
\draw[edge] (2,8) -- (4,8);
\draw[edge] (2,0) -- (2,8);
\draw[edge] (4,0) -- (4,8);
\draw[edge] (6,0) -- (6,4);
\draw[edge] (8,0) -- (8,4);
\draw[edge] (10,0) -- (10,2);
\draw[red, line width=5pt] (10.6,1) -- (11.4,1) ;
\draw[red, line width=5pt] (4.6,5) -- (5.4,5) ;
\end{tikzpicture}%
\captionsetup{labelformat=empty}
\captionof{figure}{$(5,3,2,1)$}
\end{minipage}
\begin{minipage}{.18\textwidth}
\centering
\begin{tikzpicture}[x=\unitsize,y=\unitsize,baseline=0]
\tikzset{vertex/.style={}}%
\tikzset{edge/.style={very thick}}%
\draw[dotted] (0,0) -- (14,0);
\draw[dotted] (0,2) -- (14,2);
\draw[dotted] (0,4) -- (14,4);
\draw[dotted] (0,6) -- (14,6);
\draw[dotted] (0,8) -- (14,8);
\draw[dotted] (2,-2) -- (2,10);
\draw[dotted] (4,-2) -- (4,10);
\draw[dotted] (6,-2) -- (6,10);
\draw[dotted] (8,-2) -- (8,10);
\draw[dotted] (10,-2) -- (10,10);
\draw[dotted] (12,-2) -- (12,10);
\draw[edge] (2,0) -- (10,0);
\draw[edge] (2,2) -- (10,2);
\draw[edge] (2,4) -- (8,4);
\draw[edge] (2,6) -- (4,6);
\draw[edge] (2,8) -- (4,8);
\draw[edge] (2,0) -- (2,8);
\draw[edge] (4,0) -- (4,8);
\draw[edge] (6,0) -- (6,4);
\draw[edge] (8,0) -- (8,4);
\draw[edge] (10,0) -- (10,2);
\draw[red, line width=5pt] (10.6,1) -- (11.4,1) ;
\draw[red, line width=5pt] (4.6,5) -- (5.4,5) ;
\draw[red, line width=5pt] (5,7) -- (7,7) -- (7,5) ;
\end{tikzpicture}%
\captionsetup{labelformat=empty}
\captionof{figure}{$(5,3,3,3)$}
\end{minipage}

\begin{minipage}{.18\textwidth}
\centering
\begin{tikzpicture}[x=\unitsize,y=\unitsize,baseline=0]
\tikzset{vertex/.style={}}%
\tikzset{edge/.style={very thick}}%
\draw[dotted] (0,0) -- (14,0);
\draw[dotted] (0,2) -- (14,2);
\draw[dotted] (0,4) -- (14,4);
\draw[dotted] (0,6) -- (14,6);
\draw[dotted] (0,8) -- (14,8);
\draw[dotted] (2,-2) -- (2,10);
\draw[dotted] (4,-2) -- (4,10);
\draw[dotted] (6,-2) -- (6,10);
\draw[dotted] (8,-2) -- (8,10);
\draw[dotted] (10,-2) -- (10,10);
\draw[dotted] (12,-2) -- (12,10);
\draw[edge] (2,0) -- (10,0);
\draw[edge] (2,2) -- (10,2);
\draw[edge] (2,4) -- (8,4);
\draw[edge] (2,6) -- (4,6);
\draw[edge] (2,8) -- (4,8);
\draw[edge] (2,0) -- (2,8);
\draw[edge] (4,0) -- (4,8);
\draw[edge] (6,0) -- (6,4);
\draw[edge] (8,0) -- (8,4);
\draw[edge] (10,0) -- (10,2);
\draw[red, line width=5pt] (10.6,1) -- (11.4,1) ;
\draw[red, line width=5pt] (4.6,5) -- (5.4,5) ;
\draw[red, line width=5pt] (5,7) -- (7,7) -- (7,5) -- (9,5) -- (9,3) ;
\end{tikzpicture}%
\captionsetup{labelformat=empty}
\captionof{figure}{$(5,4,4,3)$}
\end{minipage}
\begin{minipage}{.18\textwidth}
\centering
\begin{tikzpicture}[x=\unitsize,y=\unitsize,baseline=0]
\tikzset{vertex/.style={}}%
\tikzset{edge/.style={very thick}}%
\draw[dotted] (0,0) -- (14,0);
\draw[dotted] (0,2) -- (14,2);
\draw[dotted] (0,4) -- (14,4);
\draw[dotted] (0,6) -- (14,6);
\draw[dotted] (0,8) -- (14,8);
\draw[dotted] (2,-2) -- (2,10);
\draw[dotted] (4,-2) -- (4,10);
\draw[dotted] (6,-2) -- (6,10);
\draw[dotted] (8,-2) -- (8,10);
\draw[dotted] (10,-2) -- (10,10);
\draw[dotted] (12,-2) -- (12,10);
\draw[edge] (2,0) -- (10,0);
\draw[edge] (2,2) -- (10,2);
\draw[edge] (2,4) -- (8,4);
\draw[edge] (2,6) -- (4,6);
\draw[edge] (2,8) -- (4,8);
\draw[edge] (2,0) -- (2,8);
\draw[edge] (4,0) -- (4,8);
\draw[edge] (6,0) -- (6,4);
\draw[edge] (8,0) -- (8,4);
\draw[edge] (10,0) -- (10,2);
\draw[red, line width=5pt] (10.6,1) -- (11.4,1) ;
\draw[red, line width=5pt] (8.6,3) -- (9.4,3) ;
\end{tikzpicture}%
\captionsetup{labelformat=empty}
\captionof{figure}{$(5,4,1,1)$}
\end{minipage}
\begin{minipage}{.18\textwidth}
\centering
\begin{tikzpicture}[x=\unitsize,y=\unitsize,baseline=0]
\tikzset{vertex/.style={}}%
\tikzset{edge/.style={very thick}}%
\draw[dotted] (0,0) -- (14,0);
\draw[dotted] (0,2) -- (14,2);
\draw[dotted] (0,4) -- (14,4);
\draw[dotted] (0,6) -- (14,6);
\draw[dotted] (0,8) -- (14,8);
\draw[dotted] (2,-2) -- (2,10);
\draw[dotted] (4,-2) -- (4,10);
\draw[dotted] (6,-2) -- (6,10);
\draw[dotted] (8,-2) -- (8,10);
\draw[dotted] (10,-2) -- (10,10);
\draw[dotted] (12,-2) -- (12,10);
\draw[edge] (2,0) -- (10,0);
\draw[edge] (2,2) -- (10,2);
\draw[edge] (2,4) -- (8,4);
\draw[edge] (2,6) -- (4,6);
\draw[edge] (2,8) -- (4,8);
\draw[edge] (2,0) -- (2,8);
\draw[edge] (4,0) -- (4,8);
\draw[edge] (6,0) -- (6,4);
\draw[edge] (8,0) -- (8,4);
\draw[edge] (10,0) -- (10,2);
\draw[red, line width=5pt] (10.6,1) -- (11.4,1) ;
\draw[red, line width=5pt] (8.6,3) -- (9.4,3) ;
\draw[red, line width=5pt] (4.6,5) -- (5.4,5) ;
\end{tikzpicture}%
\captionsetup{labelformat=empty}
\captionof{figure}{$(5,4,2,1)$}
\end{minipage}
\begin{minipage}{.18\textwidth}
\centering
\begin{tikzpicture}[x=\unitsize,y=\unitsize,baseline=0]
\tikzset{vertex/.style={}}%
\tikzset{edge/.style={very thick}}%
\draw[dotted] (0,0) -- (14,0);
\draw[dotted] (0,2) -- (14,2);
\draw[dotted] (0,4) -- (14,4);
\draw[dotted] (0,6) -- (14,6);
\draw[dotted] (0,8) -- (14,8);
\draw[dotted] (2,-2) -- (2,10);
\draw[dotted] (4,-2) -- (4,10);
\draw[dotted] (6,-2) -- (6,10);
\draw[dotted] (8,-2) -- (8,10);
\draw[dotted] (10,-2) -- (10,10);
\draw[dotted] (12,-2) -- (12,10);
\draw[edge] (2,0) -- (10,0);
\draw[edge] (2,2) -- (10,2);
\draw[edge] (2,4) -- (8,4);
\draw[edge] (2,6) -- (4,6);
\draw[edge] (2,8) -- (4,8);
\draw[edge] (2,0) -- (2,8);
\draw[edge] (4,0) -- (4,8);
\draw[edge] (6,0) -- (6,4);
\draw[edge] (8,0) -- (8,4);
\draw[edge] (10,0) -- (10,2);
\draw[red, line width=5pt] (10.6,1) -- (11.4,1) ;
\draw[red, line width=5pt] (4.6,5) -- (5.4,5) ;
\draw[red, line width=5pt] (5,7) -- (7,7) -- (7,5) ;
\draw[red, line width=5pt] (8.6,3) -- (9.4,3) ;
\end{tikzpicture}%
\captionsetup{labelformat=empty}
\captionof{figure}{$(5,4,3,3)$}
\end{minipage}
\end{exmp}

\subsection{The BGG correspondence.} \label{subsec:BGG}
Recall from Section~\ref{subsection:LieSuperalgebra} that $S = \Sym(V) = \bbC[x_{i,j}]$ and $E = \bigwedge W = \bbC \langle e_{i,j} \rangle$. If $M = \dsum_{t \in \bbZ} M_t$ is a finitely generated graded $S$-module, let $M^{\vee}$ denote its \emph{graded dual},
\[
\DS M^{\vee} = \dsum_{t \in \bbZ} \Hom_{\bbC} (M_t, \bbC) = \dsum_{t \in \bbZ} M_t^{\vee},
\]
where the action of $S$ is given by $(s \smcdot \phi)(m) = \phi(s \smcdot m)$ for $s \in S$, $\phi \in M^{\vee}$ and $m \in M$ homogeneous elements. We can associate to $M$ a complex $\tilde{\bR}(M)$ of free $E$-modules (which is a modification of the complex $\tilde{\bR}(M)$ in \cite{Eis05}[Section 7E]):
\[
\tilde{\bR}(M) \colon \quad \ldots \to E \otimes M_t^{\vee} \to E \otimes M_{t-1}^{\vee} \to \ldots
\]
\[
\DS e \otimes \phi \mapsto \sum_{i,j} e \cdot e_{i,j} \otimes x_{i,j} \cdot \phi.
\]
The convention we use here is that $\deg(E_s) = \deg(\bigwedge^s W) = s$. So we grade $E$ positively such that all the $W$-variables $e_{i,j}$ get degree $1$. This is different from \cite{Eis05}[Section 7B] where the $W$-variables are given degree $-1$. With this convention we give $M_t^{\vee}$ degree $t$ and an analogue of \cite{Eis05}[Proposition 7.21] gives us
\[
H_t(\tilde{\bR}(M))_{s+t} \simeq \Tor_s(\bbC,M)_{s+t}^{\vee}
\]
The $E$-module $H_t(\tilde{\bR}(M))$ is finitely generated and it encodes (up to taking vector space duals) the $t$-th linear strand of the minimal free resolution of $M$.

For simplicity of notations, we assume $M = I_{\lambda} \subseteq S$ is the principal $\GL$-equivariant ideal generated by $S_{\lambda} V_0 \otimes S_{\lambda} V_1$. Then $M$ is a $\gl(V_0) \times \gl(V_1)$-equivariant $S$-module with $M_t = \dsum_{\substack{|\mu| = t \\ \lambda \leq \mu}} S_{\mu} V_0 \otimes S_{\mu} V_1$. And
\[
\DS \tilde{\bR}(I_{\lambda})_t = E \otimes M_t^{\vee} = \dsum_{\substack{|\mu| = t \\ \lambda \leq \mu}} E \otimes (S_{\mu} W_0 \otimes S_{\mu} W_1) = \dsum_{\substack{|\mu| = t \\ \lambda \leq \mu}} K_{\mu}.
\]
It is proven that in \cite{RW}[Theorem 3.1] that this makes $\tilde{\bR}(I_{\lambda})$ a complex of $\fg$-modules. Moreover, the map $K_{\mu} \to K_{\nu}$ in $\tilde{\bR}(I_{\lambda})_t \to \tilde{\bR}(I_{\lambda})_{t-1}$ is nonzero if and only if $\nu \leq \mu$ and $|\mu| = |\nu| + 1$ (Pictorially, this means $\mu$ and $\nu$ are only differed by one box). And we will use this in Section ~\ref{section:MainProof} to analyze $H_t(\tilde{\bR}(I_{\lambda}))$ as a $\fg$-module.

\section{Main Proof} \label{section:MainProof}
\begin{defn}
Let $V$ be a finite dimensional $\fg$-module. We say $v_{\lambda}$ is a \emph{primitive vector} of weight $\lambda$ in $V$ if there exists submodule $W \subseteq V$ such that $\fg_{+1} v \in W$ but $v \notin W$. In this case, we also say that $\lambda$ is a \emph{primitive weight} of $V$. Two primitive vectors which generate the same indecomposable submodule are considered the same. 
\end{defn}

\begin{notn}
The multiplicity $m_{\lambda}$ of a primitive weight $\lambda$ is equal to the dimension of the subspace generated by all primitive vectors in $V$ of weight $\lambda$. Also $v_{\lambda}^{\mu}$ denotes a primitive vector of weight $\lambda$ in the Kac module $K_{\mu}$.
\end{notn}

Via BGG correspondence, we can concentrate on the map $\psi \colon K_{\mu} \to K_{\lambda}$ where $|\mu| = |\lambda| + 1$ and $\lambda \leq \mu$. By Lemma 3.2 \cite{RW}, we have this map is nonzero and unique up to scaling. Note that there is a unique primitive vector $v_{\mu}^{\lambda}$ in $K_{\lambda}$. We have that $\psi(K_{\mu}) \subseteq K_{\lambda}$ is on the one hand, the submodule generated by the primitive vector $v_{\mu}^{\lambda} \in K_{\lambda}$. On the other hand, it is isomorphic to a quotient of $K_{\mu}$. We also have the fact that the multiplicity $m_{\lambda}$ of any primitive weight $\lambda$ in any Kac module is $1$. Hence a primitive weight uniquely determines its corresponding primitive vector. Also the primitive vectors generated by $v_{\mu}^{\lambda}$ in $K_{\lambda}$ is a subset of the primitive vectors in $K_{\mu}$. In other words, if $v$ is a primitive vector in $K_{\mu}$ and $\psi(v) \neq 0$, then $\psi(v)$ is a primitive vector of the same weight in $\psi(K_{\mu}) \subseteq K_{\lambda}$. And therefore a primitive vector in $K_{\lambda}$. 

\begin{lem} \label{lemma:image1}
Let $\alpha$ be a primitive weight whose corresponding primitive vector is generated by $v_{\mu}^{\lambda}$. And let $v_{\alpha}^{\mu}$ be the primitive vector in $K_{\mu}$ of weight $\alpha$. Then $\psi(v_{\alpha}^{\mu}) = c v_{\alpha}^{\lambda}$ for some non-zero scalar $c \in \bbC$. In particular, $\psi(v_{\alpha}^{\mu}) \neq 0$.
\end{lem}

\begin{proof}
By the last paragraph, we only need to show that $\psi(v_{\alpha}^{\mu}) \neq 0$. In fact, if $\psi(v_{\alpha}^{\mu}) = 0$, then the submodule generated by $v_{\alpha}^{\mu}$ will be in $\ker \psi$. If we let $[M]$ denote the class of a $\fg$-module $M$ in the Grothendieck group $K_0(\fg)$ of finite dimensional representations of $\fg$, then the coefficient of $[\bbL_{\alpha}]$ is $1$ in $[\ker \psi]$. But the coefficient of $[\bbL_{\alpha}]$ is $1$ in $[K_{\mu}]$. Hence the coefficient of $[\bbL_{\alpha}]$ is $0$ in $[\psi(K_{\mu})]$, which is isomorphic to the submodule generated by $v_{\mu}^{\lambda}$ in $K_{\lambda}$. This contradicts the fact the $v_{\mu}^{\lambda}$ generates the primitive vector $v_{\alpha}^{\lambda}$. 
\end{proof}

Now we identify primitive weights of $K_{\mu}$ (resp. $K_{\lambda}$) with $\mu$-admissible (resp. $\lambda$-admissible) Dyck patterns. Assume $\mu = \lambda(\bbD)$ where $\bbD = (D_0) = (\{(a,b)\})$ and $(a,b)$ is a corner of $\mu$. 

\begin{lem} \label{lemma:image}
Assume that $\alpha$ is a primitive weight of $K_{\mu}$. So $\alpha = \mu(\bbD)$ where $\bbD = (D_1,\ldots,D_r)$ is a $\mu$-admissible Dyck pattern. On the one hand, if  $\bbD_1 = (D_0,D_1,\ldots,D_r)$ is a $\lambda$-admissible Dyck pattern where $\alpha = \lambda(\bbD_1)$, then we have that $\psi(v_{\mu(\bbD)}^{\mu}) = v_{\lambda(\bbD_1)}^{\lambda}$.  On the other hand, if $(D_0,D_1,\ldots,D_r)$ is not a $\lambda$-admissible Dyck patterns, then $\psi(v_{\mu(\bbD)}^{\mu}) = 0$. 
\end{lem}

\begin{proof}
First of all, it follows from \cite{SZ}[Theorem 5.18] that a primitive weight $\lambda(\bbD)$ where $\bbD = (D_1,\ldots,D_r)$ is generated by the associated primitive vector of $\mu = \lambda((\{ (a,b) \}))$ if and only if $D_i = \{ (a,b) \}$ for some $i$. Without lost of generality, we assume $D_1 = \{ (a,b) \}$.  Hence if $\bbD_1$ is a $\lambda$-admissible Dyck pattern, then by Lemma \ref{lemma:image1}, we have $\psi(v_{\mu(\bbD)}^{\mu}) = v_{\lambda(\bbD_1)}^{\lambda}$. Otherwise, if $\bbD_1$ is not $\lambda$-admissible and $\psi(v_{\mu(\bbD)}^{\mu}) \neq 0$, then $\psi(v_{\mu(\bbD)}^{\mu}) = v_{\lambda(\bbD_2)}^{\lambda}$ where $\bbD_2 = (\{ (a,b) \},D_2',\ldots,D_l')$ is a $\lambda$-admissible Dyck pattern and $\alpha = \lambda(\bbD_2)$. Now we look at the $\mu$-admissible Dyck pattern $\bbD_3 = (D_2',\ldots,D_l')$. We have $\alpha = \mu(\bbD) = \mu(\bbD_3)$. Then in $K_{\mu}$, the multiplicity of $\alpha$ will be $m_{\alpha} \geq 2$, which contradicts the fact that the multiplicity of any primitive weight in any Kac module is $1$. Therefore $\psi(v_{\mu(\bbD)}^{\mu}) = 0$.
\end{proof}

For the simplicity of notations, we assume that $I = I_{\lambda}$ is the principal $\GL$-invariant ideal in the coordinate ring of $\bbC^{m \times n}$ generated by a single summand $S_{\lambda} \bbC^m \otimes S_{\lambda} \bbC^n$. Via BGG correspondence, $\tilde{\bR}(I)_i = \dsum_{\substack{|\mu| = i \\ \lambda \leq \mu}} K_{\mu}$.  Now we can put the unique Loewy filtration \cite{BS} (the definition of Loewy filtration could be found in  ~\ref{defn:Loewy}) on all the Kac modules in $\tilde{\bR}(I)$. In other words, let $r = min\{m,n\}$, for $K_{\mu} \in \tilde{\bR}(I)$, assume $\{0\} = K_{\mu}^0 \subset K_{\mu}^1 \subset \ldots \subset K_{\mu}^r \subset K_{\mu}^{r+1} = K_{\mu} $ is the unique Loewy filtration on $K_{\mu}$. The $p$-th filtration subgroup of $\tilde{\bR}(I)_i$  is $F^p(\tilde{\bR}(I)_i) = \dsum_{\substack{|\mu| = i \\ \lambda \leq \mu}} K_{\mu}^p$. This makes $\tilde{\bR}(I)$ a filtered complex of $\fg$-modules. If we look at the spectral sequence associated to this filtered complex, the main proposition is as follows:

\begin{prop} \label{prop:degenerate}
The above spectral sequence degenerates on the second page.
\end{prop}

Before the proof of the above proposition, we will need the following lemma first.

\begin{lem} \label{lemma:exact}
Suppose we have a complex of vector space $\bF_{\bullet}$ where $F_i = \bC^{\binom{n}{i}} = \dsum_{\substack{A \subseteq \{ 1, \ldots, n \} \\ |A| = i}} \bbC_A $ where the map $F_i \to F_{i+1}$ has the property that the map $\bbC_A \to \bbC_B$ is nonzero if and only if $A \subseteq B$. Then the complex $\bF_{\bullet}$ is exact.
\end{lem}


\begin{proof}
We do induction on $n$. When $n = 1$, this is trivial. Let $0 \to \bK_{\bullet} \to \bF_{\bullet} \to \bQ_{\bullet} \to 0$ be a short exact sequence where $Q_i = \dsum_{\substack{A \subseteq \{1, \ldots, n-1 \} \\ |A| = i}} \bbC_A$. Then $\bQ_{\bullet}$ is exact by the induction hypothesis. Now $K_i = \dsum_{\substack{\{n\} \subseteq A \subseteq \{1, \ldots, n \} \\ |A| = i}} \bbC_A = \dsum_{\substack{A' \subseteq \{ 1,\ldots,n-1 \} \\ |A'| = i-1}} \bbC_{A'}$ where $\bbC_{A'} \to \bbC_{B'}$ is nonzero if and only if $A \subseteq B$ if and only if $A' \subseteq B'$. Hence by the induction hypothesis, $\bK_{\bullet}$ is exact. By the associated long exact sequence on homology, we have $\bF_{\bullet}$ is also exact.
\end{proof}



\begin{proof}[Proof of Proposition~\ref{prop:degenerate}]
 Note that all objects in $E^r_{p,q}$ are semisimple $\fg$-modules $\forall p,q,r$. Furthermore, on the first page of this spectral sequence, all primitive weight vectors becomes highest weight vectors. The differential in $\tilde{\bR}(I)$ will decrease the filtration index by $1$. 
 
 On the first page, we concentrate on $\bbL_{\mu(\bbD)} \in E^1_{p,|\mu|-p}$ where $\lambda \leq \mu$ and $\bbD = (D_1,,\ldots,D_k,\ldots,D_l)$ is a $\mu$-admissible Dyck pattern. Assume $|D_i| = 1$ for all $1 \leq i \leq k$ and $|D_j| \geq 3$ for all $j > k$. Define $\mu_0 = \mu(\bbD ')$ where $\bbD ' = (D_1,\ldots,D_k)$. Then $\mu_0(\bbD_0) = \mu(\bbD)$ where $\bbD_0 = (D_{k+1},\ldots,D_{l})$ is a $\mu_0$-admissible Dyck pattern. Define $\mu_n$ to be the minimal tableaux such that $\lambda \leq \mu_n$ and $\mu_n(\bbD_n) = \mu(\bbD)$ where $\bbD_n = (D_{k+1},\ldots,D_{l},D_1 ',\ldots,D_n ')$ is a $\mu_n$-admissible Dyck pattern and $|D_i '| = 1$ for all $1 \leq i \leq n$. Note that we can assume that $D_i ' = D_i$ for all $1 \leq i \leq k$. Each $D_i '$ will consist of a corner of $\mu_0$. In other words, $(D_i ') \mid_{1 \leq i \leq n}$ is the maximal collection of size $1$ Dyck paths $\{ (a,b) \}$ such that $(a,b) \in \mu_0$ and exact one of the two situations happens: 
 \begin{enumerate}
     \item No box located directly N,E,NE from $(a,b)$ is in $\mu_0(\bbD_0)$.
     \item There exists an $i$ with $k+1 \leq i \leq l$ such that all boxes located directly N,E,NE from $(a,b)$ are in $D_i$.
 \end{enumerate}
 Let $A \subseteq \{1,\ldots,n\}$. Define $\mu_A = \mu_n(\bbD_{A^c}')$ where $\bbD_{A^c} ' = (D_i ') \mid_{\substack{i \notin A \\ 1 \leq i \leq n}}$ is a $\mu_n$-admissible Dyck pattern. Also define $\bbD_A = (D_{k+1},\ldots,D_l,D_i ') \mid_{\substack{i \in A \\ 1 \leq i \leq n}}$. This is a $\mu_A$-admissible Dyck pattern. Using this notation, we have $\bbD = \bbD_{\{1,\ldots,k\}}$ and $\mu = \mu_{\{1,\ldots,k\}}$. 
 
Here is an example where $m=n=4$, $\mu = (4,3,1^2)$, $\lambda = (3,2,1)$ and $\mu(\bbD) = (5^2,2,1)$. In this example, we have $n=3$.

 \begin{minipage}{.50\textwidth}
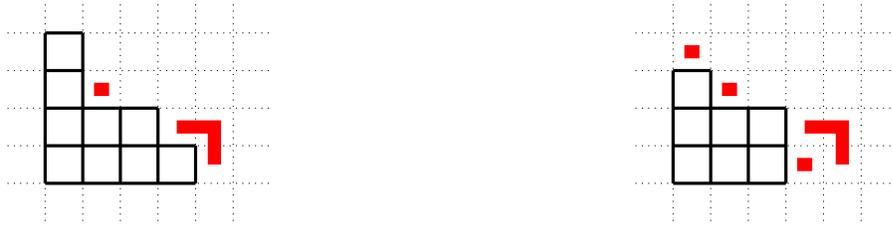
\centering
\begin{tikzpicture}[x=\unitsize,y=\unitsize,baseline=0]
\tikzset{vertex/.style={}}%
\tikzset{edge/.style={very thick}}%
\draw[dotted] (0,0) -- (14,0);
\draw[dotted] (0,2) -- (14,2);
\draw[dotted] (0,4) -- (14,4);
\draw[dotted] (0,6) -- (14,6);
\draw[dotted] (0,8) -- (14,8);
\draw[dotted] (2,-2) -- (2,10);
\draw[dotted] (4,-2) -- (4,10);
\draw[dotted] (6,-2) -- (6,10);
\draw[dotted] (8,-2) -- (8,10);
\draw[dotted] (10,-2) -- (10,10);
\draw[dotted] (12,-2) -- (12,10);
\draw[edge] (2,0) -- (10,0);
\draw[edge] (2,2) -- (10,2);
\draw[edge] (2,4) -- (8,4);
\draw[edge] (2,6) -- (4,6);
\draw[edge] (2,8) -- (4,8);
\draw[edge] (2,0) -- (2,8);
\draw[edge] (4,0) -- (4,8);
\draw[edge] (6,0) -- (6,4);
\draw[edge] (8,0) -- (8,4);
\draw[edge] (10,0) -- (10,2);
\draw[red, line width=5pt] (9,3) -- (11,3) -- (11,1) ;
\draw[red, line width=5pt] (4.6,5) -- (5.4,5) ;
\end{tikzpicture}%
\captionsetup{labelformat=empty}
\captionof{figure}{$\mu = \mu_{\{1\}} = (4,3,1,1) \\ \mu(\bbD) = \mu_{\{1\}}(\bbD_{\{1\}}) = (5,5,2,1)$}
\end{minipage}
 \begin{minipage}{.50\textwidth}
\centering
\begin{tikzpicture}[x=\unitsize,y=\unitsize,baseline=0]
\tikzset{vertex/.style={}}%
\tikzset{edge/.style={very thick}}%
\draw[dotted] (0,0) -- (14,0);
\draw[dotted] (0,2) -- (14,2);
\draw[dotted] (0,4) -- (14,4);
\draw[dotted] (0,6) -- (14,6);
\draw[dotted] (0,8) -- (14,8);
\draw[dotted] (2,-2) -- (2,10);
\draw[dotted] (4,-2) -- (4,10);
\draw[dotted] (6,-2) -- (6,10);
\draw[dotted] (8,-2) -- (8,10);
\draw[dotted] (10,-2) -- (10,10);
\draw[dotted] (12,-2) -- (12,10);
\draw[edge] (2,0) -- (8,0);
\draw[edge] (2,2) -- (8,2);
\draw[edge] (2,4) -- (8,4);
\draw[edge] (2,6) -- (4,6);
\draw[edge] (2,0) -- (2,6);
\draw[edge] (4,0) -- (4,6);
\draw[edge] (6,0) -- (6,4);
\draw[edge] (8,0) -- (8,4);
\draw[red, line width=5pt] (9,3) -- (11,3) -- (11,1) ;
\draw[red, line width=5pt] (4.6,5) -- (5.4,5) ;
\draw[red, line width=5pt] (8.6,1) -- (9.4,1) ;
\draw[red, line width=5pt] (2.6,7) -- (3.4,7) ;
\end{tikzpicture}%
\captionsetup{labelformat=empty}
\captionof{figure}{$\mu_n = \mu_3 = \mu_{\{1,2,3\}} = (3,3,1) \\ \mu_{\{1,2,3\}}(\bbD_{\{1,2,3\}}) = (5,5,2,1)$}
\end{minipage}

 According to Lemma~\ref{lemma:image}, the component $\bbL_{\mu}(\bbD)$ fits into the following direct summand of the first page:
 \begin{align*}
 0 \to \bbL_{\mu_{\emptyset}(\bbD_{\emptyset})} \to \dsum_{\substack{1 \leq i \leq n}} \bbL_{\mu_{\{i\}}(\bbD_{\{i\}})} \to \ldots 
 \to \dsum_{\substack{|A| = i \\ A \subseteq \{1,\ldots,n\}}} \bbL_{\mu_A(\bbD_A)} \to \ldots \to \bbL_{\mu_{\{1,\ldots,n\}}(\bbD_{\{1,\ldots,n\}})} \to 0.      
 \end{align*}
Moreover, if $|A| = i$ and $|B| = i+1$, then $\bbL_{\mu_A(\bbD_A)} \to \bbL_{\mu_B(\bbD_B)}$ is nonzero if and only if $\mu_B \leq \mu_A$ if and only if $A \subseteq B$. Hence according to Lemma~\ref{lemma:exact}, this complex is exact if and only of $n \geq 1$. Therefore, if $n \geq 1$, the component $\bbL_{\mu(\bbD)}$ will not appear on the second page of the spectral sequence.

In the case when $n = 0$, first of all, we have that $\bbD$ is a $\mu$-admissible Dyck pattern with all Dyck paths of size at least $3$. Furthermore, for any $\nu$ with $|\nu| = |\mu| - 1$ and $\nu \leq \mu$, assume $\mu = \nu(\bbD_{\nu})$ where $\bbD_{\nu} = (D_{\nu})$ and $|D_{\nu}| = 1$ , $\bbD_1 := (D_1,\ldots,D_l,D_{\nu})$ is either not a $\nu$-admissible Dyck pattern or $\lambda \nleq \nu$. According to Lemma~\ref{lemma:image}, the corresponding primitive vector of $\mu(\bbD)$ maps to $0$ in the original complex $\tilde{R}(I_{\lambda})$. Hence the corresponding $\fg$-module $\bbL_{\mu(\bbD)}$ will map to $0$ on the first and second page in the spectral sequence. And all differentials on the second page of this spectral sequence is $0$. Therefore the spectral sequence degenerates on the second page.
\end{proof}

Now we restate the main conjecture in \cite{RW} and give a proof here.

Let $\lambda$ be a partition with $l(\lambda) \leq n$. We consider the set of $\lambda$-admissible augmented Dyck patterns
\begin{align*}
    \bbD = (D_1,\ldots,D_r ; \bbB)
\end{align*}
with no Dyck path of length one, and $l(\lambda(\bbD)) \leq n$.
\begin{align*}
\cA(\lambda;n) = \{ \bbD \text{ a $\lambda$-admissible Dyck pattern} \mid |D_i| \geq 3 \text{ for all } 1 \leq i \leq r, \text{ and } \lambda(\bbD)_j = 0 \text{ for } j > n   \}
\end{align*}

\begin{thm} \label{thm: main}
Suppose $m \geq n$ are positive integers, $S$ is the coordinate ring of $\bbC^{m \times n}$, $\lambda$ is a partition with at most $n$ parts, and $I_{\lambda} \subseteq S$ is the corresponding principal $\GL$-equivariant ideal. For $b \geq 0$ we have the following equality in the Grothendieck group $K_0(\fg)$ of finite dimensional representations of $\fg$.
\begin{align}
    \DS [H_{|\lambda|+b}(\tilde{\bR}(I_{\lambda}))] = \sum_{\substack{\bbD \in \cA(\lambda;n) \\ b(\bbD) = b}} [\bbL_{\lambda(\bbD)}].
\end{align}
\end{thm}

\begin{proof}
We use the previous spectral sequence to compute $[H_{|\lambda|+b}(\tilde{\bR}(I_{\lambda}))]$. Assume $\bbD' = (D_1,\ldots,D_l)$ and define
\begin{align*}
\cB(\lambda;b,n) = \{ (\mu,\bbD') \mid & \mu \text{ is a partition with } \lambda \leq \mu, |\mu| - |\lambda| = b, l(\mu) \leq n, l(\mu(\bbD')) \leq n. \\
& \bbD' \text{ is a $\mu$-admissible Dyck pattern with } |D_i| \geq 3 \text{ for all } 1 \leq i \leq l. \\
& \text{If } (a,b) \in \mu \text{ is a corner of } \mu, \text{ such that  it satisfies one of } \\
& \text{the following two conditions:} \\ 
&  \text{1. } \exists D_i \in \bbD' \text{ such that all the boxes located directly N,E or NE from } \\
& \text{ $(a,b)$ are in $D_i$.} \\
&  \text{2. There is no box located directly N,E or NE from $(a,b)$ lies in supp$(\bbD')$.} \\ 
& \text{Then } (a,b) \in \lambda \}.
\end{align*}
From the last paragraph of the proof of Proposition~\ref{prop:degenerate}, we have that on the second page of the spectral sequence,
\begin{align*}
\DS E^2_{|\lambda|+b} = \dsum_{p + q = |\lambda|+b} E^2_{p,q} = \dsum_{(\mu,\bbD') \in \cB(\lambda;b,n)} \bbL_{\mu(\bbD')}.    
\end{align*}
And the spectral sequence degenerates on the second page, so we have
\begin{align}
\DS    [H_{|\lambda|+b}(\tilde{\bR}(I_{\lambda}))] = \sum_{(\mu,\bbD') \in \cB(\lambda;b,n)} [\bbL_{\mu(\bbD')}].
\end{align}
The proof will be finished by the following Lemma~\ref{lem:1-1}.
\end{proof}

\begin{lem} \label{lem:1-1}
There is a 1-1 correspondence between the multi-sets
\begin{align}
\xymatrix{
    \{ \lambda(\bbD) \mid \bbD \in \cA(\lambda;n), b(\bbD) = b \} &  \longleftrightarrow & \{ \mu(\bbD') \mid (\mu,\bbD') \in \cB(\lambda;b,n)  \}.
}
\end{align}
\end{lem}

\begin{proof}
Let $\bbD = (D_1,\ldots,D_r;\bbB) \in \cA(\lambda;n)$ and $b(\bbD) = b$. The map goes from left to right sends $\lambda(\bbD)$ to $\mu(\bbD')$ where $\mu = \lambda(\bbB)$ and $\bbD' = (D_1,\ldots,D_r)$. Then it is clear that $\mu$ is a partition with $(\mu,\bbD') \in \cB(\lambda;b,n)$. The map goes from right to left sends $\mu(\bbD')$ where $\bbD' = (D_1,\ldots,D_l)$ to $\lambda(\bbD)$ where $\bbD = (D_1,\ldots,D_l;\bbB = \mu \setminus \lambda)$. The only thing that needs to check here is that $\bbD = (D_1,\ldots,D_l;\bbB = \mu \setminus \lambda)$ is an augmented $\lambda$-admissible Dyck pattern. 

We will order the corners of $\mu$ based on the number of the first coordinate. So $(a_i,b_i) < (a_j,b_j)$ if and only if $a_i < a_j$. Assume $(a_0,b_0) < (a_1,b_1) < \ldots < (a_n,b_n)$ are all the corners of $\mu$. If we consider any $(a_k,b_k)$ which is a corner of $\mu$ such that $(a_k,b_k) \notin \lambda$, then one of the two situations happens:
\begin{enumerate} \label{situations}
    \item There exists $1 \leq i \leq l$ such that $(a_k,b_k+1),(a_k,b_k+2) \in D_i$. But $(a_k+1,b_k+1),(a_k+1,b_k) \notin D_i$.
    \item There exists $1 \leq i \leq l$ such that $(a_k + 1,b_k),(a_k+2,b_k) \in D_i$. But $(a_k+1,b_k+1),(a_k,b_k+1) \notin D_i$.
\end{enumerate}
These two situations might not be disjoint. We will focus on the first one since the second one could be dealt with similarly. If we are in the first situation, assume $(a_m,b_m)$ is the smallest corner of $\mu$ such that $(a_m,b_m) > (a_k,b_k)$ and $(a_m,b_m) \in \lambda$. Pick $B_i$ to be $B_i = \{ (a_k,b_k),(a_k,b_k-1),(a_k,b_k-2),\ldots,(a_k,b_m+1) \} \subseteq \mu \setminus \lambda$, then $(D_i,B_i)$ is an augmented Dyck path. For any $a_{k-1} < c \leq a_k$, if we consider $(c,b_k) \in \mu \setminus \lambda$, we will be in the same situation as $(a_k,b_k)$. In other words, $\exists 1 \leq j_c \leq l$ such that $(c,b_k+1,c,b_k+2) \in D_{j_c}$. But $(c+1,b_k+1),(c+1,b_k) \notin D_{j_c}$. Hence if we pick $B_{j_c} = \{ (c,b_k),(c,b_k-1),(c,b_k-2),\ldots,(c,b_m+1) \} \subseteq \mu \setminus \lambda$, then $(D_{j_c},B_{j_c})$ is an augmented Dyck path. Note that this forces $(a_{k-1},b_{k-1})$ to be either in $\lambda$ or in the first situation. In this way, we break $\bbB = \mu \setminus \lambda$ into subsets (might not be disjoint) $B_j$ such that $(D_j,B_j)$ is a Dyck path for all $j$. Hence we have $\bbD = (D_1,\ldots,D_l;\bbB)$ is a $\lambda$-admissible Dyck pattern.
\end{proof}

\begin{rmk}
Suppose $m \geq n$, for any $\GL$-invariant ideal $I$ in the coordinate ring of $\bbC^{m \times n}$, we can write $I = I_{\lambda_1} + I_{\lambda_2} + \ldots + I_{\lambda_N}$ where $I_{\lambda_i}$ is the principal $\GL$-invariant ideal generated by $S_{\lambda_i} \bbC^n \otimes S_{\lambda_i} \bbC^m$. Furthermore, we assume $I_{\lambda_i}$ and $I_{\lambda_j}$ do not generate each other if $i \neq j$. Or equivalently, we assume $\lambda_i$ and $\lambda_j$ are not comparable if $i \neq j$. 

Consider the following natural long exact sequence: 
\[
0 \longleftarrow I_{\lambda_1} + I_{\lambda_2} + \ldots + I_{\lambda_N} \longleftarrow \dsum_{i,j} I_{\lambda_{i,j}} \longleftarrow \dsum_{i,j,k} I_{\lambda_{i,j,k}} \longleftrightarrow \ldots \longleftarrow I_{\lambda_{1,2,\ldots,N}} \longleftarrow 0
\]
where $\lambda_{i_1,\ldots,i_k}$ is the union of $\lambda_{i_1}$ through $\lambda_{i_k}$. The above long exact sequence gives us an associated spectral sequence relating $\dsum_{\substack{\{i_1,\ldots,i_k\} \subset \{1,\ldots,N\} \\ 1 \leq k \leq N}} \Tor_{\bullet}^S(I_{\lambda_{i_1,\ldots,i_k}},\bbC)$ and $\Tor^S_{\bullet}(I,\bbC)$. A similar argument as in the proof of Proposition~\ref{prop:degenerate} gives us the cancellation of the terms in the spectral sequence. From there, we can get $\Tor^S_{\bullet}(I,\bbC)$ as $\gl(m|n)$ modules. On the other hand, there is no close formula to get $\Tor^S_{\bullet}(I,\bbC)$ for all $\GL$ invariant ideals $I$ as $\gl(m|n)$ modules, unlike the principal invariant case.
\end{rmk}

\begin{rmk}
We can think of each of the $\fg$-modules $\bbL_{\lambda_i(\bbD)}$ as giving rise to a linear complex appearing as a subquotient in the minimal free resolution of $I$ via BGG correspondence. More precisely, $\bbL_{\lambda_i(\bbD)}$ corresponds to a linear complex that appears entirely within the row indexed by $|\lambda_i| + b(\bbD)$ of the Betti table, starting in column $d(\bbD)$.
\end{rmk}

We give an example of how to use this to compute Betti table of a $\GL$-invariant ideal $I_{\lambda}$. This example overlaps the one in \cite{RW}[Example 4.2].

\begin{exmp}
Consider $m=n=3$ and $\lambda = (3,2)$. The Dyck patterns in $\cA(\lambda = (3,2);3)$ are as follows (labelled by $\lambda(\bbD)$):

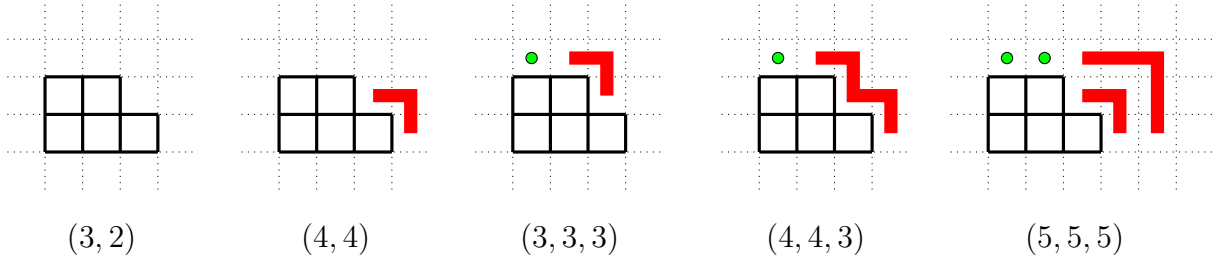
\begin{minipage}{.18\textwidth}
\centering
\begin{tikzpicture}[x=\unitsize,y=\unitsize,baseline=0]
\tikzset{vertex/.style={}}%
\tikzset{edge/.style={very thick}}%
\draw[dotted] (0,0) -- (10,0);
\draw[dotted] (0,2) -- (10,2);
\draw[dotted] (0,4) -- (10,4);
\draw[dotted] (0,6) -- (10,6);
\draw[dotted] (2,-2) -- (2,8);
\draw[dotted] (4,-2) -- (4,8);
\draw[dotted] (6,-2) -- (6,8);
\draw[dotted] (8,-2) -- (8,8);
\draw[edge] (2,0) -- (8,0);
\draw[edge] (2,2) -- (8,2);
\draw[edge] (2,4) -- (6,4);
\draw[edge] (2,0) -- (2,4);
\draw[edge] (4,0) -- (4,4);
\draw[edge] (6,0) -- (6,4);
\draw[edge] (8,0) -- (8,2);
\end{tikzpicture}%
\captionsetup{labelformat=empty}
\captionof{figure}{$(3,2)$}
\end{minipage}
\begin{minipage}{.18\textwidth}
\centering
\begin{tikzpicture}[x=\unitsize,y=\unitsize,baseline=0]
\tikzset{vertex/.style={}}%
\tikzset{edge/.style={very thick}}%
\draw[dotted] (0,0) -- (10,0);
\draw[dotted] (0,2) -- (10,2);
\draw[dotted] (0,4) -- (10,4);
\draw[dotted] (0,6) -- (10,6);
\draw[dotted] (2,-2) -- (2,8);
\draw[dotted] (4,-2) -- (4,8);
\draw[dotted] (6,-2) -- (6,8);
\draw[dotted] (8,-2) -- (8,8);
\draw[edge] (2,0) -- (8,0);
\draw[edge] (2,2) -- (8,2);
\draw[edge] (2,4) -- (6,4);
\draw[edge] (2,0) -- (2,4);
\draw[edge] (4,0) -- (4,4);
\draw[edge] (6,0) -- (6,4);
\draw[edge] (8,0) -- (8,2);
\draw[red, line width=5pt] (7,3) -- (9,3) -- (9,1) ;
\end{tikzpicture}%
\captionsetup{labelformat=empty}
\captionof{figure}{$(4,4)$}
\end{minipage}
\begin{minipage}{.18\textwidth}
\centering
\begin{tikzpicture}[x=\unitsize,y=\unitsize,baseline=0]
\tikzset{vertex/.style={}}%
\tikzset{edge/.style={very thick}}%
\draw[dotted] (0,0) -- (10,0);
\draw[dotted] (0,2) -- (10,2);
\draw[dotted] (0,4) -- (10,4);
\draw[dotted] (0,6) -- (10,6);
\draw[dotted] (2,-2) -- (2,8);
\draw[dotted] (4,-2) -- (4,8);
\draw[dotted] (6,-2) -- (6,8);
\draw[dotted] (8,-2) -- (8,8);
\draw[edge] (2,0) -- (8,0);
\draw[edge] (2,2) -- (8,2);
\draw[edge] (2,4) -- (6,4);
\draw[edge] (2,0) -- (2,4);
\draw[edge] (4,0) -- (4,4);
\draw[edge] (6,0) -- (6,4);
\draw[edge] (8,0) -- (8,2);
\draw[red, line width=5pt] (5,5) -- (7,5) -- (7,3);
\draw[fill=green] (3,5) circle [radius=0.3] ;
\end{tikzpicture}%
\captionsetup{labelformat=empty}
\captionof{figure}{$(3,3,3)$}
\end{minipage}
\begin{minipage}{.2\textwidth}
\centering
\begin{tikzpicture}[x=\unitsize,y=\unitsize,baseline=0]
\tikzset{vertex/.style={}}%
\tikzset{edge/.style={very thick}}%
\draw[dotted] (0,0) -- (10,0);
\draw[dotted] (0,2) -- (10,2);
\draw[dotted] (0,4) -- (10,4);
\draw[dotted] (0,6) -- (10,6);
\draw[dotted] (2,-2) -- (2,8);
\draw[dotted] (4,-2) -- (4,8);
\draw[dotted] (6,-2) -- (6,8);
\draw[dotted] (8,-2) -- (8,8);
\draw[edge] (2,0) -- (8,0);
\draw[edge] (2,2) -- (8,2);
\draw[edge] (2,4) -- (6,4);
\draw[edge] (2,0) -- (2,4);
\draw[edge] (4,0) -- (4,4);
\draw[edge] (6,0) -- (6,4);
\draw[edge] (8,0) -- (8,2);
\draw[red, line width=5pt] (5,5) -- (7,5) -- (7,3) -- (9,3) -- (9,1) ;
\draw[fill=green] (3,5) circle [radius=0.3] ;
\end{tikzpicture}%
\captionsetup{labelformat=empty}
\captionof{figure}{$(4,4,3)$}
\end{minipage}
\begin{minipage}{.2\textwidth}
\centering
\begin{tikzpicture}[x=\unitsize,y=\unitsize,baseline=0]
\tikzset{vertex/.style={}}%
\tikzset{edge/.style={very thick}}%
\draw[dotted] (0,0) -- (14,0);
\draw[dotted] (0,2) -- (14,2);
\draw[dotted] (0,4) -- (14,4);
\draw[dotted] (0,6) -- (14,6);
\draw[dotted] (2,-2) -- (2,8);
\draw[dotted] (4,-2) -- (4,8);
\draw[dotted] (6,-2) -- (6,8);
\draw[dotted] (8,-2) -- (8,8);
\draw[dotted] (10,-2) -- (10,8);
\draw[dotted] (12,-2) -- (12,8);
\draw[edge] (2,0) -- (8,0);
\draw[edge] (2,2) -- (8,2);
\draw[edge] (2,4) -- (6,4);
\draw[edge] (2,0) -- (2,4);
\draw[edge] (4,0) -- (4,4);
\draw[edge] (6,0) -- (6,4);
\draw[edge] (8,0) -- (8,2);
\draw[red, line width=5pt] (7,3) -- (9,3) -- (9,1);
\draw[red, line width=5pt] (7,5) -- (11,5) -- (11,1);
\draw[fill=green] (3,5) circle [radius=0.3] ;
\draw[fill=green] (5,5) circle [radius=0.3] ;
\end{tikzpicture}%
\captionsetup{labelformat=empty}
\captionof{figure}{$(5,5,5)$}
\end{minipage}

According to Theorem ~\ref{thm: main}, we have:
\begin{align*}
    & [H_5(\tilde{\bR}(I_{(3,2)}))] = [\bbL_{(3,2)}] + [\bbL_{(4,4)}] \\
    & [H_6(\tilde{\bR}(I_{(3,2)}))] = [\bbL_{(3,3,3)}] + [\bbL_{(4,4,3)}] \\
    & [H_7(\tilde{\bR}(I_{(3,2)}))] = [\bbL_{(5,5,5)}] 
\end{align*}
We can use \cite{SZ07}[Section 4] to compute the Hilbert series of the graded $E$-module $\bbL_{\mu}$: if we write $\mathrm{HS}_{\mu}(t)$ for the Hilbert series of the graded $E$-module $\bbL_{\mu}$ then 
\begin{align*}
    & \mathrm{HS}_{(3,2)}(t) = 225 t^5 + 1132 t^6 + 2673 t^7 + 3582 t^8 + 2785 t^9 + 1188 t^{10} + 225 t^{11},   \\
    & \mathrm{HS}_{(4,4)}(t) = 225 t^8 + 700 t^9 + 828 t^{10} + 450 t^{11} + 100 t^{12}, \\
    & \mathrm{HS}_{(3,3,3)}(t) = t^9, \quad \mathrm{HS}_{(4,4,3)}(t) = 9 t^{11} + 16 t^{12} + 9 t^{13}, \quad \mathrm{HS}_{(5,5,5)}(t) = t^{15}.
\end{align*}
We can also use the above computation to get the following betti table of $I_{(3,2)}$:
\begin{equation}\label{eq:betti-I32}
\begin{matrix}
     &0&1&2&3&4&5&6&7&8\\
     \text{5:}&225&
     1132&2673&3807&3485&2016&675&100&\text{.}\\\text{6:}&\text{.}&\text{.}&\text{.}&1&\text{.}&
     9&16&9&\text{.}\\\text{7:}&\text{.}&\text{.}&\text{.}&\text{.}&\text{.}&\text{.}&\text{.}&\text{.}&1\\
\end{matrix}
\end{equation}
\end{exmp}

\small \noindent Hang Huang, Department of Mathematics,
University of Wisconsin Madison, Madison, WI 53706 \\
{\tt hhuang235@math.wisc.edu}, \url{http://math.wisc.edu/~hhuang235/}

\end{document}